\documentclass[11pt,reqno,twoside]{article}
\usepackage{amsmath,amsfonts,amsthm,amssymb}
\usepackage{graphics}
\usepackage[all]{xy}
\usepackage{color}
\theoremstyle{definition}
\newtheorem{theorem}{Theorem}[section]
\newtheorem{proposition}[theorem]{Proposition}

\newtheorem{lemma}[theorem]{Lemma}
\newtheorem{remark}[theorem]{Remark}

\newtheorem{example}[theorem]{Example}
\newtheorem{conjecture}[theorem]{Conjecture}
\begin{document}

\title{\normalsize\bf CROSSING NUMBER OF AN ALTERNATING KNOT AND CANONICAL GENUS OF ITS WHITEHEAD DOUBLE}

\author{\small 
HEE JEONG JANG and SANG YOUL LEE
\smallskip\\
{\small\it 
Department of Mathematics, Pusan National University,
}\\ 
{\small\it Busan 609-735, Korea}\\
{\small\it heejeong@pusan.ac.kr}\\
{\small\it sangyoul@pusan.ac.kr}}

\renewcommand\leftmark{\centerline{\footnotesize 
H. J. Jang and S. Y. Lee}}
\renewcommand\rightmark{\centerline{\footnotesize 
Crossing number of an alternating knot and canonical genus}}

\maketitle

\begin{abstract}
A conjecture proposed by J. Tripp in 2002 states that the crossing number of any knot coincides with the canonical genus of its Whitehead double. In the meantime, it has been established that this conjecture is true for a large class of alternating knots including $(2, n)$ torus knots, $2$-bridge knots, algebraic alternating knots, and alternating pretzel knots. In this paper, we prove that the conjecture is not true for any alternating $3$-braid knot which is the connected sum of two torus knots of type $(2, m)$ and $(2, n)$. This results in a new modified conjecture that the crossing number of any prime knot coincides with the canonical genus of its Whitehead double. We also give a new large class of prime alternating knots satisfying the conjecture, including all prime alternating $3$-braid knots.
\end{abstract}

\noindent{\it Mathematics Subject Classification 2000}: 57M25; 57M27.

\noindent{\it Key words and phrases}: Alternating knot; $3$-braid knot; canonical genus; crossing number; Morton's inequality; Whitehead double; Tripp's conjecture.

%%%%

\section{Introduction}\label{sec:intro}

In 2002, J. Tripp \cite{Tri} proved that the canonical genus of a Whitehead double of a torus knot $T(2,n)$ of type $(2,n)$ is equal to $n$, the crossing number of $T(2,n)$. To prove this, he used Morton's inequality \cite{Mot} and verified that the maximal $z$-degree $\max\deg_z P_{W_\pm(T(2,n),m)}(v,z)$ of the HOMFLYPT polynomial of the positive/negative $m$-twisted Whitehead double $W_\pm(T(2,n),m)$ of $T(2,n)$ is equal to two times of the crossing number $c(T(2,n))$, i.e., $\max\deg_z P_{W_\pm(T(2,n),m)}(v,z)=2c(T(2,n))$, which implies immediately the result. Motivating this, he conjectured the following: 

\begin{conjecture}\cite{Tri}\label{Tripp-conj-0}
The crossing number of any knot coincides with the canonical genus of its Whitehead double.
\end{conjecture}
 
In \cite{Nak}, T. Nakamura had extended Tripp's argument to show that Conjecture \ref{Tripp-conj-0} for $2$-bridge knots holds, and proposed the following:

\begin{conjecture}\cite{Nak}\label{Nakam-conj-0}
For any alternating knot $K$ of crossing number $c(K)$, we have $\max\deg_z P_{W_\pm(K,m)}(v,z)=2c(K)$. Therefore the canonical genus of a Whitehead double of $K$ is equal to $c(K)$.
\end{conjecture}

He also showed that Conjecture \ref{Nakam-conj-0} for a non-alternating knot (actually the torus knot of type $(4,3)$) is false.

In \cite{BJ}, M. Brittenham and J. Jensen showed that Conjecture \ref{Nakam-conj-0} holds for alternating pretzel knots $P(k_1,\ldots,k_n), k_1,\ldots, k_n \geq 1$ \cite[Theorem 1]{BJ}. To prove this, they provided a method of building new knots $K$ with 
$\max\deg_z P_{W_\pm(K,m)}(v,z)$ $=2c(K)$ from old ones $K'$ (For more details, see \cite[Section 3]{Hj} or \cite{BJ}).
Actually, Brittenham and Jensen gave a larger class of alternating knots than the class of $(2,n)$-torus knots, $2$-bridge knots, and alternating pretzel knots.
In addition, H. Gruber \cite{Gr1} extended Nakamura's result to algebraic alternating knots in Conway's sense in a different way. Quite recently, the authors \cite{Hj} gave a new infinite family of alternating knots for which Conjecture \ref{Nakam-conj-0} holds, which is an extension of the previous results of Tripp \cite{Tri}, Nakamura \cite{Nak} and Brittenham-Jensen \cite{BJ}.

\bigskip

For $n\geq 2$, let $B_{n}$ denote the $n$-strand (geometric) braid group which has a group presentation whose generators are $\sigma_1, \sigma_2, \ldots, \sigma_{n-1}$ as shown in Fig.~\ref{fig-generators} and defining relations are:
\begin{align*}
&\sigma_i\sigma_j=\sigma_j\sigma_i ~\text{if}~ |i-j| \geq 2, 1 \leq i, j \leq n-1;\\
&\sigma_i\sigma_{i+1}\sigma_i
=\sigma_{i+1}\sigma_{i}\sigma_{i+1},
1 \leq i \leq n-2.
\end{align*}

\begin{figure}
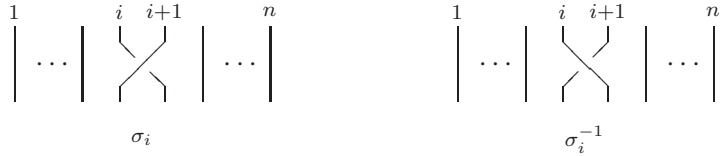
%[t]
%\vspace*{5pt}
\centerline{ \xy (20,5);(20,7) **@{-},
(20,13);(20,15) **@{-}, (20,7);(26,13) **@{-}, (20,13);(22,11)
**@{-}, (24,9);(26,7) **@{-}, (26,5);(26,7) **@{-},
(26,13);(26,15) **@{-},
(15,5);(15,15) **@{-}, (6,5);(6,15) **@{-}, (31,5);(31,15) **@{-},
(40,5);(40,15) **@{-}, (11,10)*{\cdots}, (36,10)*{\cdots},
(6,17)*{_{1}}, (20,17)*{_{i}}, (40,17)*{_{n}}, (26,17)*{_{i+1}},
(23,0)*{_{\sigma_i}},
\endxy
  \qquad\qquad\qquad
\xy (20,5);(20,7) **@{-}, (20,13);(20,15) **@{-}, (20,13);(26,7)
**@{-}, (20,7);(22,9) **@{-}, (24,11);(26,13) **@{-},
(26,5);(26,7) **@{-}, (26,13);(26,15) **@{-},
(15,5);(15,15) **@{-}, (6,5);(6,15) **@{-}, (31,5);(31,15) **@{-},
(40,5);(40,15) **@{-}, (11,10)*{\cdots}, (36,10)*{\cdots},
(6,17)*{_{1}}, (20,17)*{_{i}}, (40,17)*{_{n}}, (26,17)*{_{i+1}},
(23,0)*{_{\sigma_i^{-1}}},
\endxy}
{}\vspace*{5pt}\caption{The standard generators of $B_{n}$}
\label{fig-generators}
\end{figure}

The {\it product} $ab$ of two braids $a$ and $b$ in $B_n$ is obtained by putting them end to end and rescaling. An element of $B_n$ is called an {\it $n$-braid}. The {\it closure} of an $n$-braid $b \in B_n$ is the link, denoted by $\hat b$, obtained by connecting the upper points of its strands to the lower ones by $n$ disjoint arcs, and is sometimes called a {\it closed braid}. As is well known, any link $L$ is the closure of a braid $b \in B_n$ for some $n\geq 2$. In this case, we say that $b$ represents $L$ or $b$ is a (braid) representative of $L$. The minimum number of braid strings needed to represent a link $L$ is called the {\it braid index} of the link $L$. For more details, we refer to \cite{BZ,Kaw}.

The class of all knots and links of braid index $3$ is a very special class, like the class of the torus knots and links, the class of the 2-bridge knots and links, the class of the algebraic knots and links, and the class of the pretzel knots and links, etc. These special classes of knots and links are rich enough to serve as a source of examples on which a researcher may be able to test various conjectures \cite{BM}. As already mentioned above Conjecture \ref{Nakam-conj-0} holds for alternating knots belong to the latter four classes and so does Conjecture \ref{Tripp-conj-0}. In this paper, we are going to test Conjectures \ref{Tripp-conj-0} and \ref{Nakam-conj-0} for alternating knots of braid index $3$.

K. Murasugi \cite{Mur2} and A. Stoimenow \cite{Sto} gave classifications of alternating links of braid index $3$. We recall Stoimenow's theorem for our convenience.
We call an $n$-braid $\beta=\sigma_{i_1}^{\epsilon_1} \cdots \sigma_{i_k}^{\epsilon_k}, \epsilon_i=\pm 1, 1 \leq i_1, \ldots, i_k \leq$ $n-1,$ an {\it alternating braid} if $\epsilon_j=\epsilon_\ell$ iff $i_j \equiv i_\ell \pmod 2$. For a positive integer $k$, the {\it $(2, k)$-torus link} is just the closure of $2$-braid $\sigma_1^k \in B_2$.

\begin{theorem}\cite[Theorem 4]{Sto}\label{ch-alk-b3}
Let $L$ be an alternating link of braid index $3$. Then (and only then) $L$ is
\begin{itemize}
\item [(i)] the connected sum of two $(2, k)$-torus links (with parallel orientation), or
\item [(ii)] an alternating $3$-braid link (i.e., the closure of an alternating $3$-braid, including split unions of a $(2,k)$-torus link and an unknot and the $3$ component unlink), or
\item [(iii)] a pretzel link $P(1,p,q,r)$ with $p,q,r \geq 1$(oriented so that the twists corresponding to $p,q,r$ are parallel).
\end{itemize}
\end{theorem}

In this paper, we prove the following.

\begin{theorem}\label{main-thm1}
For each pair $i, j$ of odd integers $\geq 3$, let $K_i$ and $K_j$ denote the $(2, i)$- and $(2,j)$-torus knot, respectively, and let $K_{i,j}=K_i \sharp K_j$, the connected sum of $K_i$ and $K_j$, which is an alternating knot of braid index $3$. For any integer $m$, let $g_c(W_{\pm}(K_{i,j},m))$ denote the canonical genus of the $m$-twisted positive/negative Whitehead double $W_{\pm}(K_{i,j},m)$ of $K_{i,j}$. Then
\begin{equation*}
g_c(W_{\pm}(K_{i,j},m))=i+j-1=c(K_{i,j})-1.
\end{equation*}
\end{theorem}

\begin{theorem}\label{main-thm2}
Let $K$ be an alternating knot of braid index $3$, which is not the connected sum of $(2, k)$-torus knot and $(2, k')$-torus knot with $k, k'\geq 3$. Then the crossing number of $K$ coincides with the canonical genus $g_c(W_{\pm}(K,m))$ of its $m$-twisted positive/negative Whitehead double $W_{\pm}(K,m)$ for any integer $m$. That is,
\begin{equation*}
g_c(W_{\pm}(K,m))=c(K).
\end{equation*}
\end{theorem}

Theorem \ref{main-thm1} shows that Conjecture \ref{Tripp-conj-0} is not true for composite (alternating) knots in general (cf. Remark \ref{conj-rmk-1}). As a conclusion, it is reasonable to propose the following:

\begin{conjecture}\label{jl-conj}
The crossing number of any prime knot coincides with the canonical genus of its Whitehead double.
\end{conjecture}

Furthermore, Lemma \ref{lem-max-d-lij} in Section \ref{sect-1-braid} below shows that Conjecture \ref{Nakam-conj-0} is also not true for composite alternating knots in general (cf. Remark \ref{conj-rmk-1}). Hence we have

\begin{conjecture}\label{jl-conj-2}
For any prime alternating knot $K$ of crossing number $c(K)$, we have $\max\deg_z P_{W_\pm(K,m)}(v,z)=2c(K)$. Therefore the canonical genus of a Whitehead double of $K$ is equal to $c(K)$.
\end{conjecture}

It is worth pointing out that Conjectures \ref{jl-conj} and \ref{jl-conj-2} are both true for prime alternating knots lie in the four special classes mentioned above. Additionally, the following theorem \ref{main-thm3} supplies a larger class of (prime) alternating knots than the class of all (prime) alternating knots with braid index $3$, for which Conjecture \ref{jl-conj-2} (and consequently Conjecture \ref{jl-conj}) holds.

\begin{theorem}\label{main-thm3}
Let $\gamma_p=(\sigma_2^{\epsilon}\sigma_{1}^{-\epsilon})^p,~\epsilon=\pm 1, p \geq 2,$ be an alternating $3$-braid and let $\overline{\mathcal K}_p$ be the class consisting of the alternating knot $\hat\gamma_p$ itself (if it is a knot)
and all alternating knots having diagrams which can be obtained from the diagram of the closed braid $\hat\gamma_p$ as shown in Fig.~\ref{fig-gamma-p} by repeatedly replacing a crossing by a full twist.
Then for every $K \in \overline{\mathcal K}_p$ and every integer $m$, \begin{equation}\label{eqn0-thm-2}
\max\deg_z P_{W_\pm(K,m)}(v,z)=2c(K),
\end{equation}
and therefore $$g_c(W_\pm(K,m))=c(K).$$
\end{theorem}

In \cite{Hj}, the authors gave a family $\mathcal K^3=\bigcup_{p=1}^\infty\mathcal K_p$ of alternating knots, where $\mathcal K_1$ contains all $(2,n)$-torus knots, $2$-bridge knots and alternating pretzel knots and $\mathcal K_i \not=\mathcal K_j$ if $i\not= j$, and showed that the crossing number of any alternating knot in $\mathcal K^3$ coincides with the canonical genus of its Whitehead double. This leads that Conjectures \ref{jl-conj} and \ref{jl-conj-2} hold for the infinite family $\mathcal K^3_{\rm prime}$ of all prime alternating knots in $\mathcal K^3$.

We remark that Theorem \ref{main-thm3} gives an infinite sequence
\[\overline{\mathcal K}_2, \overline{\mathcal K}_3, \ldots, \overline{\mathcal K}_p, \ldots\]
of infinite families $\overline{\mathcal K}_p$ of (prime) alternating knots satisfying Conjecture \ref{jl-conj-2} and therefore Conjecture \ref{jl-conj}. We define
\[\overline{\mathcal K}^2=\bigcup_{p=2}^\infty\overline{\mathcal K}_p.\]
Then the infinite family $\overline{\mathcal K}^2_{\rm prime}$ of all prime alternating knots in $\overline{\mathcal K}^2$ is a new family that supports Conjectures \ref{jl-conj} and \ref{jl-conj-2}, including all prime alternating knots with braid index $3$, and also containing infinitely many prime alternating knots with braid index $> 3$ (see Example \ref{example}). Therefore Conjectures \ref{jl-conj} and \ref{jl-conj-2} hold for all prime alternating knots that belong to the family $\mathcal K^{32}_{\rm prime}:=\mathcal K^3_{\rm prime}$ $\cup$ $\overline{\mathcal K}^2_{\rm prime}$. We also note that $\mathcal K^{32}_{\rm prime}$ provides a partial affirmative answer to the conjecture given by Brittenham and Jensen in the last section 4 of the paper \cite{BJ}, which states that if $K$ is a nontrivial prime alternating knot, then
$\max\deg_z P_{W_\pm(K,m)}(v,z)=2c(K),$
and thus $g_c(W_\pm(K,m))=c(K).$
It is remarkable from Proposition \ref{BJprop2-cr-nbr-cg-wd} below that
if $K'$ is a knot belong to $\mathcal K^{32}_{\rm prime}$ and if for a $c(K')$-minimizing diagram $D'$ for $K'$ we replace a crossing of $D'$, thought of as a half-twist, with three half-twists as shown in Fig.~\ref{fig:3-half twists}, producing a new alternating knot $K$, then we also have $\max\deg_z P_{W_\pm(K,m)}(v,z)=2c(K),$ and therefore $g_c(W_\pm(K,m))=c(K)$.

 \bigskip

The rest of this paper is organized as follows.
Section 2 consists of definitions and terminologies which are used throughout this paper. Indeed, we review the Morton's inequality for the maximum degree in $z$ of the HOMFLYPT polynomial $P_L(v,z)$ of a link $L$ and its  relation to the canonical genus of Whitehead double of a knot. We also give a brief review of Brittenham and Jensen's results from \cite{BJ,Hj}. In Section 3, we prove Theorem \ref{main-thm1}.
In Section 4, we prove that for all integers $p \geq 2$, the maximum degree in $z$ of the HOMFLYPT polynomial $P_{W_2(\hat\gamma_p)}(v,z)$ of the doubled link $W_2(\hat\gamma_p)$ of the closure $\hat\gamma_p$ of an alternating $3$-braid $\gamma_p=(\sigma_2^{\epsilon}\sigma_{1}^{-\epsilon})^p,~\epsilon=\pm 1, p \geq 2,$ is equal to $2c(\hat\gamma_p)-1=4p-1$ (Theorem \ref{main-cor-1}). Using this result and Brittenham and Jensen's results, we prove Theorem \ref{main-thm2} and Theorem \ref{main-thm3} in Section 5 and discuss examples. 
 The final section 6 is devoted to prove a key lemma \ref{m-d-1}, which has an essential role to prove Theorem \ref{main-cor-1}.

\section{Terminologies and notations}\label{sect-wd}

Let $D$ be an oriented diagram of an oriented knot $K$ and let $w(D)$
denote the writhe of $D$, that is, the sum of the signs of all
crossings in $D$ defined by
$\mathrm{sign} \left( \xy (5,2.5);(0,-2.5) **@{-} ?<*\dir{<},
(5,-2.5);(3,-0.5) **@{-}, (0,2.5);(2,0.5) **@{-} ?<*\dir{<},
\endxy \right) = 1$ and $\mathrm{sign} \left( \xy (0,2.5);(5,-2.5)
**@{-} ?<*\dir{<}, (0,-2.5);(2,-0.5) **@{-}, (5,2.5);(3,0.5)
**@{-} ?<*\dir{<}, \endxy \right) = -1$.
Recall that for an oriented diagram $D = D_1 \cup D_2$ of an oriented
two component link $L = K_1 \cup K_2$, the {\it linking number}
$lk(L)$ of $L$ is defined to be the half of the sum of the signs of all
crossings between $D_1$ and $D_2$.

Let $T$ be a knot embedded in the unknotted solid torus $V=S^1 \times D^2$, which is essential in the sense that it meets every meridional disc in the solid torus $V$. Let $K$ be an arbitrary given knot in $S^3$
and let $N(K)$ be a tubular neighborhood of $K$ in $S^3$.
Suppose that $h : V=S^1 \times D^2 \to N(K)$ is a homeomorphism,
then the image $h(T)=S_T(K)$ is a new knot in $S^{3}$, which is called a
{\it satellite (knot)} with the {\it companion} $K$ and {\it pattern} $T$, and denoted by $S_T(K)$.
Note that if $K$ is a non-trivial knot, then $S_T(K)$ is also a non-trivial knot \cite{BZ}.

Now let $W_+$, $W_-$ and $U$ denote the positive Whitehead-clasp,
negative Whitehead-clasp and the doubled link embedded in $V$
with orientations as shown in Fig.~ \ref{Whitehead-clasp}.

\begin{figure}[ht]
\begin{center}
\resizebox{0.60\textwidth}{!}{%
  \includegraphics{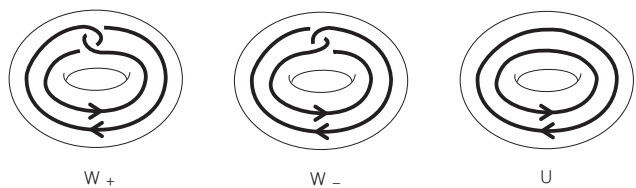} }
\caption{Whitehead-clasp}
\label{Whitehead-clasp}
\end{center}
\end{figure}

Let $K$ be an oriented knot and let $h : V=S^1 \times D^2 \to N(K)$ be an orientation preserving homeomorphism which takes
the disk $\{\mathbf{1}\} \times D^2$ to a meridian disk of $N(K)$,
and the core $S^1 \times \{\mathbf{0}\}$ of $V$ onto the knot $K$.
Let $\ell$ be the preferred longitude of $V$. We choose an orientation for the image $h(\ell)$ so that it is parallel to $K$.
If the linking number of $h(\ell)$ and $K$ is equal to $m$,
then the satellite $S_{W_+}(K)$ (respectively  $S_{W_-}(K)$)
with the companion $K$ and pattern $W_+$ (respectively  $W_-$) is called the {\it $m$-twisted positive} (respectively  {\it negative}) {\it Whitehead
double} of $K$, denoted by $W_+(K, m)$(respectively  $W_-(K, m)$), and the satellite $S_U(K)$ with the companion $K$ and pattern $U$ is called
the {\it $m$-twisted doubled link} of $K$, denoted by $W_2(K,m)$.
The $0$-twisted positive (respectively  negative) Whitehead double of $K$ is
sometimes called the {\it untwisted} positive (respectively  negative)
Whitehead double of $K$.
In what follows, we use the notation $W_\pm(K, m)$ to refer to the
$m$-twisted positive/negative Whitehead double of $K$.

\begin{figure}
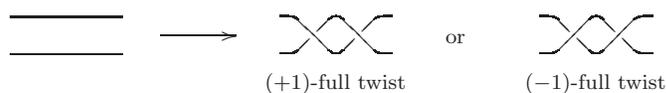

\centerline{ \xy (0,10);(15,10) **@{-},
(0,15);(15,15) **@{-} ,
(20,12.5);(30,12.5) **@{-} ?>*\dir{>},
(7.5,5)*{_{}},
\endxy
\quad
\xy (0,10);(1,10) **@{-},
(0,15);(1,15) **@{-}, (2.5,10.5);
(4,12) **@{-}, (5,13);(6.5,14.5) **@{-}, (2.5,14.5);(6.5,10.5) **@{-},
(8.5,10.5);(10,12) **@{-}, (11,13);(12.5,14.5) **@{-},
(8.5,14.5);(12.5,10.5) **@{-}, (1,10);(2.5,10.5) **\crv{(2,10)},
(1,15);(2.5,14.5) **\crv{(2,15)}, (6.5,14.5);(8.5,14.5)
**\crv{(7.5,15.5)}, (6.5,10.5);(8.5,10.5) **\crv{(7.5,9.5)},
(14,10);(12.5,10.5) **\crv{(13,10)}, (14,15);(12.5,14.5)
**\crv{(13,15)}, (14,10);(15,10) **@{-}, (14,15);(15,15) **@{-},
(7.5,6)*{_\text{$(+1)$-full twist}}, (23.5,12)*{_{\rm or}},
\endxy
\qquad
\xy (0,10);(1,10) **@{-},
(0,15);(1,15) **@{-}, (2.5,14.5);
(4,13) **@{-}, (5,12);(6.5,10.5) **@{-}, (2.5,10.5);(6.5,14.5) **@{-},
(8.5,14.5);(10,13) **@{-}, (11,12);(12.5,10.5) **@{-},
(8.5,10.5);(12.5,14.5) **@{-}, (1,10);(2.5,10.5) **\crv{(2,10)},
(1,15);(2.5,14.5) **\crv{(2,15)}, (6.5,14.5);(8.5,14.5)
**\crv{(7.5,15.5)}, (6.5,10.5);(8.5,10.5) **\crv{(7.5,9.5)},
(14,10);(12.5,10.5) **\crv{(13,10)}, (14,15);(12.5,14.5)
**\crv{(13,15)}, (14,10);(15,10) **@{-}, (14,15);(15,15) **@{-},
(7.5,6)*{_\text{$(-1)$-full twist}},
\endxy}
\vspace*{0pt}\caption{($\pm$)-full twist}\label{fig-full-twist} 
\end{figure}

The $m$-twisted positive (respectively  negative) Whitehead double $W_+(K,m)$ (respectively  $W_-(K,m)$) has the {\it canonical diagram}, denoted by  $W_+(D,m)$
(respectively  $W_-(D,m)$), associated with a diagram $D$ of $K$,
which is the doubled link diagram of $D$ with $(m-w(D))$
full-twists (see Fig.~\ref{fig-full-twist}) and a positive Whitehead-clasp
$W_+$ (respectively  negative Whitehead-clasp $W_-$) as illustrated in
(b) and (c) of Fig.~\ref{fig-ex-canon-diag-trefoil}. Also, the $m$-twisted doubled link
$W_2(K,m)$ of $K$ has the canonical diagram $W_2(D,m)$ associated with $D$,
which is the doubled link diagram of $D$ with $(m-w(D))$ full-twists
without Whitehead-clasp.
In particular, the canonical diagram $W_+(D,w(D))$ (respectively $W_-(D,w(D))$) of the $w(D)$-twisted positive (respectively  negative) Whitehead double $W_+(K,w(D))$ (respectively  $W_-(K,w(D))$) is called the
{\it standard diagram} of Whitehead double of $K$ associated with the diagram $D$ and is denoted by simply $W_+(D)$ (respectively  $W_-(D)$). Likewise, the canonical diagram $W_2(D,w(D))$ of the $w(D)$-twisted doubled link $W_2(K,w(D))$ is called the
{\it standard diagram} of the doubled link of $K$ associated with the diagram $D$ and is denoted by simply $W_2(D)$ (For example, see Fig. \ref{fig-ex-canon-diag-trefoil} (d)).

\begin{figure}[ht]
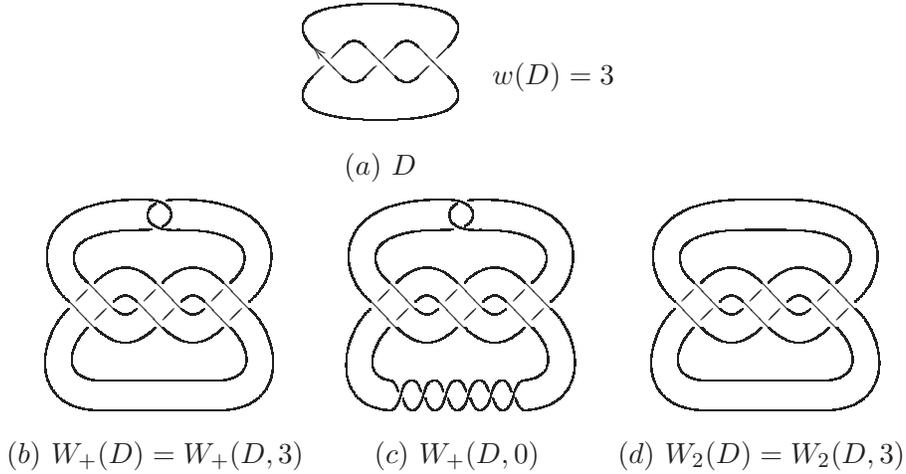

\centerline{\xy (32,28)*{(a)~D}, (55,40)*{w(D)=3},
(23,44);(27,40)
**@{-}, ?<*\dir{<}, (23,40);(24.5,41.5) **@{-},
(25.5,42.5);(27,44) **@{-},
(30,44);(34,40) **@{-}, (30,40);(31.5,41.5) **@{-},
(32.5,42.5);(34,44) **@{-},
(37,44);(41,40) **@{-},  (37,40);(38.5,41.5) **@{-},
(39.5,42.5);(41,44) **@{-},
(27,44);(30,44) **\crv{(28.5,45.5)}, (27,40);(30,40)
**\crv{(28.5,38.5)}, (34,44);(37,44) **\crv{(35.5,45.5)},
(34,40);(37,40) **\crv{(35.5,38.5)}, (23,44);(41,44)
**\crv{(21,46)&(21,50)&(43,50)&(43,46)}, (23,40);(41,40)
**\crv{(21,38)&(21,34)&(43,34)&(43,38)},
\endxy}
\vskip 0.2cm
\centerline{\xy (33,22)*{(b)~{W_+(D)=W_+(D, 3)}},
 (22.5,42.5);(24,44) **@{-},
(25,40);(26.5,41.5) **@{-}, (21.5,42.5);(25,39) **@{-},
(24,45);(27.5,41.5) **@{-},
  %%%
(27.5,42.5);(30.5,42.5) **\crv{(29,44)}, (27.5,41.5);(30.5,41.5)
**\crv{(29,40)}, (25,45);(33,45) **\crv{(29,49)}, (25,39);(33,39)
**\crv{(29,35)},
  %%%
(31.5,42.5);(33,44) **@{-}, (34,40);(35.5,41.5) **@{-},
(30.5,42.5);(34,39) **@{-}, (33,45);(36.5,41.5) **@{-},
  %%%
(36.5,42.5);(39.5,42.5) **\crv{(38,44)}, (36.5,41.5);(39.5,41.5)
**\crv{(38,40)}, (34,45);(42,45) **\crv{(38,49)}, (34,39);(42,39)
**\crv{(38,35)},
  %%%
(40.5,42.5);(42,44) **@{-}, (43,40);(44.5,41.5) **@{-},
(39.5,42.5);(43,39) **@{-}, (42,45);(45.5,41.5) **@{-},
  %%%
(31,56);(34,52.5) **\crv{(36.2,56.2)&(35.5,53)},
(32.5,52.1);(31,52) **\crv{(32,52)}, (36,52);(33,55.5)
**\crv{(31,51.8)&(31.5,54.8)}, (34.5,55.9);(36,56)
**\crv{(35,56)},
  %%%
(31,56);(21.5,42.5) **\crv{(16,56.2)&(17,46)}, (31,52);(24,45)
**\crv{(20.5,52)&(21,47.5)}, (36,56);(45.5,42.5)
**\crv{(51,56.2)&(50,46)}, (36,52);(43,45)
**\crv{(46.5,52)&(46,47.5)},
  %%%
(25,28);(21.5,41.5) **\crv{(16,28.5)&(17.5,38)}, (25,32);(24,39) **\crv{(20.5,32.5)&(21.5,36.5)},
(42,28);(45.5,41.5) **\crv{(51,28.5)&(49.5,38)}, (42,32);(43,39) **\crv{(46.5,32.5)&(45.5,36.5)},
  %%%
(25,28);(36,28) **@{-}, (25,32);(36,32) **@{-}, (36,28);(42,28)
**@{-}, (36,32);(42,32) **@{-},
\endxy
  \quad
\xy (33,22)*{(c)~{W_+(D, 0)}},
(22.5,42.5);(24,44)
**@{-}, (25,40);(26.5,41.5) **@{-}, (21.5,42.5);(25,39) **@{-},
(24,45);(27.5,41.5) **@{-},
  %%%
(27.5,42.5);(30.5,42.5) **\crv{(29,44)},
(27.5,41.5);(30.5,41.5)
**\crv{(29,40)}, (25,45);(33,45) **\crv{(29,49)}, (25,39);(33,39)
**\crv{(29,35)},
  %%%
(31.5,42.5);(33,44) **@{-}, (34,40);(35.5,41.5) **@{-},
(30.5,42.5);(34,39) **@{-}, (33,45);(36.5,41.5) **@{-},
  %%%
(36.5,42.5);(39.5,42.5) **\crv{(38,44)},
(36.5,41.5);(39.5,41.5)
**\crv{(38,40)}, (34,45);(42,45) **\crv{(38,49)}, (34,39);(42,39)
**\crv{(38,35)},
  %%%
(40.5,42.5);(42,44) **@{-}, (43,40);(44.5,41.5) **@{-},
(39.5,42.5);(43,39) **@{-}, (42,45);(45.5,41.5) **@{-},
  %%%
(31,56);(34,52.5) **\crv{(36.2,56.2)&(35.5,53)},
(32.5,52.1);(31,52) **\crv{(32,52)}, (36,52);(33,55.5)**\crv{(31,51.8)&(31.5,54.8)},
(34.5,55.9);(36,56) **\crv{(35,56)},
  %%%
(31,56);(21.5,42.5) **\crv{(16,56.2)&(17,46)}, (31,52);(24,45)
**\crv{(20.5,52)&(21,47.5)}, (36,56);(45.5,42.5)
**\crv{(51,56.2)&(50,46)}, (36,52);(43,45)
**\crv{(46.5,52)&(46,47.5)},
  %%%
(24,28);(21.5,41.5) **\crv{(16,28.5)&(17.5,38)}, (24,32);(24,39)
**\crv{(20.5,32.5)&(21.5,36.5)}, (42,28);(45.5,41.5)
**\crv{(51,28.5)&(49.5,38)}, (42,32);(43,39)
**\crv{(46.5,32.5)&(45.5,36.5)},
  %%%
(24,28);(25.5,30) **\crv{(25,28)}, (25.5,30);(27,32)
**\crv{(26,32)}, (24,32);(25.2,30.5) **\crv{(24.7,32)},
(25.8,29.5);(27,28) **\crv{(26.3,28)}, (27,28);(28.5,30)
**\crv{(28,28)}, (28.5,30);(30,32) **\crv{(29,32)},
(27,32);(28.2,30.5) **\crv{(27.7,32)}, (28.8,29.5);(30,28)
**\crv{(29.3,28)},
  %%%
(30,28);(31.5,30) **\crv{(31,28)}, (31.5,30);(33,32)
**\crv{(32,32)}, (30,32);(31.2,30.5) **\crv{(30.7,32)},
(31.8,29.5);(33,28) **\crv{(32.3,28)}, (33,28);(34.5,30)
**\crv{(34,28)}, (34.5,30);(36,32) **\crv{(35,32)},
(33,32);(34.2,30.5) **\crv{(33.7,32)}, (34.8,29.5);(36,28)
**\crv{(35.3,28)},
  %%%
(36,28);(37.5,30) **\crv{(37,28)}, (37.5,30);(39,32)
**\crv{(38,32)}, (36,32);(37.2,30.5) **\crv{(36.7,32)},
(37.8,29.5);(39,28) **\crv{(38.3,28)}, (39,28);(40.5,30)
**\crv{(40,28)}, (40.5,30);(42,32) **\crv{(41,32)},
(39,32);(40.2,30.5) **\crv{(39.7,32)}, (40.8,29.5);(42,28)
**\crv{(41.3,28)},
\endxy
%%%
\quad
\xy (33,22)*{(d)~{W_2(D)=W_2(D,3)}},
 (22.5,42.5);(24,44) **@{-},
(25,40);(26.5,41.5) **@{-}, (21.5,42.5);(25,39) **@{-},
(24,45);(27.5,41.5) **@{-},
  %%%
(27.5,42.5);(30.5,42.5) **\crv{(29,44)},
(27.5,41.5);(30.5,41.5)**\crv{(29,40)},
(25,45);(33,45) **\crv{(29,49)},
(25,39);(33,39)**\crv{(29,35)},
  %%%
(31.5,42.5);(33,44) **@{-}, (34,40);(35.5,41.5) **@{-},
(30.5,42.5);(34,39) **@{-}, (33,45);(36.5,41.5) **@{-},
  %%%
(36.5,42.5);(39.5,42.5) **\crv{(38,44)},
(36.5,41.5);(39.5,41.5)
**\crv{(38,40)}, (34,45);(42,45) **\crv{(38,49)}, (34,39);(42,39)
**\crv{(38,35)},
  %%%
(40.5,42.5);(42,44) **@{-}, (43,40);(44.5,41.5) **@{-},
(39.5,42.5);(43,39) **@{-}, (42,45);(45.5,41.5) **@{-},
  %%%
(31,56);(36,56) **@{-}, (31,52);(36,52) **@{-},
  %%%
(31,56);(21.5,42.5) **\crv{(16,56.2)&(17,46)},
(31,52);(24,45)**\crv{(20.5,52)&(21,47.5)},
(36,56);(45.5,42.5)**\crv{(51,56.2)&(50,46)},
(36,52);(43,45)**\crv{(46.5,52)&(46,47.5)},
  %%%
(25,28);(21.5,41.5) **\crv{(16,28.5)&(17.5,38)}, (25,32);(24,39) **\crv{(20.5,32.5)&(21.5,36.5)},
(42,28);(45.5,41.5) **\crv{(51,28.5)&(49.5,38)}, (42,32);(43,39) **\crv{(46.5,32.5)&(45.5,36.5)},
  %%%
(25,28);(36,28) **@{-}, (25,32);(36,32) **@{-}, (36,28);(42,28)
**@{-}, (36,32);(42,32) **@{-},
\endxy
}
 \caption{Canonical diagrams}\label{fig-ex-canon-diag-trefoil}
\end{figure}

F. Frankel and L. Pontrjagin \cite{FP} and H. Seifert \cite{Sei} introduced a
method to construct a compact orientable surface having a given oriented link as its boundary.
A {\it Seifert surface} for an oriented link $L$ in $S^3$ is a compact,
connected, and orientable surface $\Sigma$ in $S^3$ with $\partial\Sigma=L.$ The {\it genus} of an oriented link $L$, denoted by $g(L)$, is the minimum genus of any Seifert surface of $L$. For an oriented diagram $D$ of a link $L$, it is well known
that a Seifert surface for $L$ can always be obtained from $D$ by applying Seifert's algorithm \cite{Sei}. A Seifert surface for an oriented link
$L$ constructed via Seifert's algorithm for an oriented diagram $D$ of $L$ is called
the {\it canonical Seifert surface} associated with $D$ and denoted by $\Sigma(D)$. In what follows, we denote the genus $g(\Sigma(D))$ of the canonical Seifert surface $\Sigma(D)$ by $g_c(D)$.
Then the minimum genus over all canonical Seifert surfaces for $L$
is called the {\it canonical genus} of $L$ and denoted by $g_c(L)$, i.e., \[g_c(L)=\displaystyle{\underset{\text{$D$ a diagram of $L$}}{\rm min}~g_c(D)}.\]

Note that Seifert's algorithm
applied to a knot or link diagram might not produce a minimal genus Seifert surface and the following inequality holds \cite{Sei}:
\begin{equation}\label{eq-inequality-genera}
\frac{1}{2}{\rm deg}\Delta_{K}(t)\leq g(K) \leq g_c(K).
\end{equation}
Up to now, many authors have explored
knots and links for which this inequality is strict or equal, for example,
see \cite{KK,LPS,LeS3,Liv,Mor,Nak,Tri} and therein. On the other hand,
K. Murasugi \cite{Mur} proved that if $K$ is an alternating knot, then
the equality in (\ref{eq-inequality-genera}) holds. Also we have the following:

\begin{proposition}\cite[Proposition 2.1]{Hj}\label{prop1-cr-nbr-cg-wd}
Let $K$ be a non-trivial knot and let $D$ be an oriented diagram of $K$ with $c(D)=c(K)$. Then for any integer $m$,
\begin{itemize}
\item [(i)] $g_c(W_\pm(D,m)) = g_c(W_\pm(D,w(D)))$.
\item [(ii)] $g_c(W_\pm(K,m))\leq g_c(W_\pm(D,m))=c(K).$
\end{itemize}
\end{proposition}

The {\it HOMFLYPT polynomial} $P_{L}(v,z)$ (or $P(L)$ for short) of an
oriented link $L$ in $S^3$ is defined by the following three axioms:
\begin{enumerate}
\item[(i)] $P_{L}(v,z)$ is invariant under ambient isotopy of $L$.

\item[(ii)] If $O$ is the trivial knot, then $P_O(v,z)=1.$

\item[(iii)] If $L_{+}$, $L_{-}$ and $L_{0}$ have diagrams $D_{+}$,
$D_{-}$ and $D_{0}$ which differ as shown in Fig.~\ref{fig-skein-diag-1},
then $v^{-1} P_{L_{+}}(v,z) - vP_{L_{-}}(v,z) = zP_{L_{0}}(v,z).$
\end{enumerate}

\begin{figure}[t]
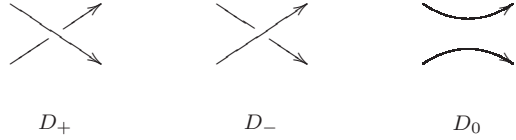

\vspace*{5pt}
\centerline{\xy (-1,14);(11,6) **@{-} ?>*\dir{>},
(11,14);(6,10.7)  **@{-} ?<*\dir{<}, (4,9.3);(-1,6) **@{-},
(5,-2)*{_{D_+}},
\endxy
 \qquad\qquad
\xy (-1,6);(11,14) **@{-} ?>*\dir{>}, (11,6);(6,9.3)  **@{-}
?<*\dir{<}, (4,10.7);(-1,14) **@{-},
(5,-2)*{_{D_-}},
\endxy
 \qquad\qquad
\xy (-1,6);(11,6) **\crv{(5,10)} ?>*\dir{>}, (-1,14);(11,14)
**\crv{(5,10)} ?>*\dir{>},
(5,-2)*{_{D_0}}, \endxy
 }
\vspace*{5pt}\caption{Skein triple}\label{fig-skein-diag-1}
\end{figure}

Let $L$ be an oriented link and let $D$ be its oriented diagram. Then $P_{L}(v,z)$ can be computed recursively by using a skein tree,
switching and smoothing crossings of $D$ until the terminal nodes are labeled with trivial links. For more details, we refer to \cite{Kaw}.
For the HOMFLYPT polynomial $P_L(v,z)$ of a link $L$, we denote the maximum degree in $z$ of $P_L(v,z)$ by $\max\deg_z P_L(v,z)$ or simply $M(L)$.

The following theorems and propositions are needed in sequel.

\begin{theorem}\cite[Theorem 2]{Mot}\label{thm-Morton-ineq}
For any oriented diagram $D$ of an oriented knot or link $L$,
\begin{equation}\label{Morton-ineq-c-gen}
\max\deg_z P_L(v,z)
\leq c(D) - s(D) + 1,
\end{equation}
where $c(D)$ is the number of crossings of $D$ and $s(D)$ is the number of the Seifert circles of $D$.
\end{theorem}

\begin{proposition}\cite[Proposition 3.1]{Hj}\label{prop3-cr-nbr-cg-wd}
Let $K$ be an oriented knot and let $D$ be an oriented diagram of $K$.
\begin{itemize}
\item [(i)] For any integer $m$ and $\epsilon=+$ or $-$, \[M(W_2(D,m)) \leq {\rm max}\{M(W_\epsilon(D,m)), 0\}-1.\]
In particular, if $M(W_\epsilon(K,m)) > 0$, then the equality holds, i.e.,
\begin{equation*}
M(W_2(D,m))=M(W_\epsilon(D,m))-1.
\end{equation*}
\item [(ii)] For any integer $m$, $M(W_2(D,w(D))) \leq {\rm max}\{M(W_2(D,m)), 1\}.$

In particular, if $M(W_2(D,w(D)))\not= 1$, then the equality holds, i.e.,
\begin{equation*}
M(W_2(D,w(D)))=M(W_2(D,m)).
\end{equation*}
\end{itemize}
\end{proposition}

\begin{proposition}\cite[Proposition 3.3]{Hj}\label{prop2-cr-nbr-cg-wd}
Let $K$ be a knot in $S^3$ with the minimal crossing number $c(K)$. If $D$ is an oriented diagram of $K$ with $c(D)=c(K)$, then for any integer $m$,
\begin{align*}
\frac{1}{2} \max\deg_z P_{W_\pm(K,m)}(v,z) &\leq g_c(W_\pm(K,m))\leq g_c(W_\pm(D,m))=c(K).
\end{align*}
\end{proposition}

\begin{proposition}\cite[Proposition 2]{BJ}\label{BJprop2-cr-nbr-cg-wd}
If $K'$ is a knot satisfying \[\max\deg_z P_{W_\pm(K',m)}(v,z)=2c(K'),\] and if for a $c(K')$-minimizing diagram $D'$ for $K'$ we replace a crossing of $D'$, thought of as a half-twist, with three half-twists as shown in Fig.~\ref{fig:3-half twists}, producing a knot $K$, then \[\max\deg_z P_{W_\pm(K,m)}(v,z)=2c(K),\] and therefore $g_c(W_\pm(K,m))=c(K)$.
\end{proposition}

\begin{figure}[t]
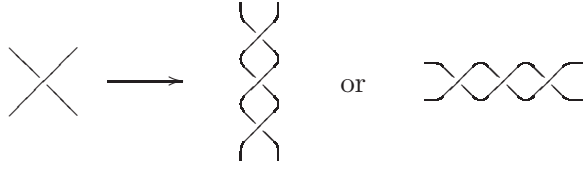

\centerline{
\xy %%%%
(12,17);(16,13) **@{-},
(17,12);(21,8) **@{-},
(12,8);(21,17) **@{-},
(25,12.5);(35,12.5) **@{-} ?>*\dir{>},
(0,2) *{},
\endxy
\quad\quad
\xy (10,0);(10,1) **@{-},
(15,0);(15,1) **@{-}, (14.5,2.5);
(13,4) **@{-}, (12,5);(10.5,6.5) **@{-}, (10.5,2.5);(14.5,6.5) **@{-},
(14.5,8.5);(13,10) **@{-},
(12,11);(10.5,12.5) **@{-},
(10.5,8.5);(14.5,12.5) **@{-},
(10,1);(10.5,2.5) **\crv{(10,2)},
(15,1);(14.5,2.5) **\crv{(15,2)},
(14.5,6.5);(14.5,8.5) **\crv{(15.5,7.5)},
(10.5,6.5);(10.5,8.5) **\crv{(9.5,7.5)},
(10,1);(10.5,2.5) **\crv{(10,2)},
(15,1);(14.5,2.5) **\crv{(15,2)},
(14.5,6.5);(14.5,8.5) **\crv{(15.5,7.5)},
(10.5,6.5);(10.5,8.5) **\crv{(9.5,7.5)},
(14.5,12.5);(14.5,14.5) **\crv{(15.5,13.5)},
(10.5,12.5);(10.5,14.5) **\crv{(9.5,13.5)},
%%%%
(14.5,14.5);(13,16) **@{-},
(12,17);(10.5,18.5) **@{-},
(10.5,14.5);(14.5,18.5) **@{-},
(10,20);(10.5,18.5) **\crv{(10,19)},
(15,20);(14.5,18.5) **\crv{(15,19)},
(10,20);(10,21) **@{-},
(15,20);(15,21) **@{-},
%%%
(25,10) *{\text{or}},
\endxy
\quad\quad
\xy (0,10);(1,10) **@{-}, 
(0,15);(1,15) **@{-}, (2.5,14.5);
(4,13) **@{-}, (5,12);(6.5,10.5) **@{-}, (2.5,10.5);(6.5,14.5) **@{-},
(8.5,14.5);(10,13) **@{-},
(11,12);(12.5,10.5) **@{-},
(8.5,10.5);(12.5,14.5) **@{-},
(1,10);(2.5,10.5) **\crv{(2,10)},
(1,15);(2.5,14.5) **\crv{(2,15)},
(6.5,14.5);(8.5,14.5) **\crv{(7.5,15.5)},
(6.5,10.5);(8.5,10.5) **\crv{(7.5,9.5)},
(1,10);(2.5,10.5) **\crv{(2,10)},
(1,15);(2.5,14.5) **\crv{(2,15)},
(6.5,14.5);(8.5,14.5) **\crv{(7.5,15.5)},
(6.5,10.5);(8.5,10.5) **\crv{(7.5,9.5)},
(12.5,14.5);(14.5,14.5) **\crv{(13.5,15.5)},
(12.5,10.5);(14.5,10.5) **\crv{(13.5,9.5)},
%%%%
(14.5,14.5);(16,13) **@{-},
(17,12);(18.5,10.5) **@{-},
(14.5,10.5);(18.5,14.5) **@{-},
(20,10);(18.5,10.5) **\crv{(19,10)},
(20,15);(18.5,14.5) **\crv{(19,15)},
(20,10);(21,10) **@{-},
(20,15);(21,15) **@{-},
%%%
(0,2) *{},
\endxy}
\vspace*{0pt}\caption{Three half-twists}\label{fig:3-half twists}
%\vskip .2cm
\end{figure}

\begin{proposition}\cite[Proposition 4]{BJ}\label{cr-1}
If $L'$ is a non-split link with a diagram $D'$ satisfying $c(D')=c(L')$ and \[\max\deg_z P_{W_2(D')}(v,z)=2c(D')-1,\] and if $L$ is a link having a diagram $D$ obtained from $D'$ by replacing a crossing in the diagram $D'$ with a full twist (so that $c(D)=c(D')+1$), then \[\max\deg_z P_{W_2(D)}(v,z)=2c(D)-1=\max\deg_z P_{W_2(D')}(v,z)+2.\]
\end{proposition}

Finally, we review Nakamura's result in \cite{Nak} about the maximum degree in $z$ of the HOMFLYPT polynomial $P_{W_2(L)}(v,z)$ of the doubled link $W_2(L)$ of a $2$-bridge link $L$, which will be used in the proof of Lemma \ref{m-d-1-1} in the section \ref{sect-miwd3-plrq}.

A {\it $2$-bridge link} $L$ is a link in $S^3$ which admits a diagram $C(a_1, a_2, \ldots, a_n)$, called Conway normal form of $L$, as shown in Fig.~\ref{fig0.1} in which each rectangle labeled $a_i$ denotes the number of half-twists with $|a_i|$ crossings as shown in Fig.~\ref{fig0.2} \cite{KaMi}.
\begin{figure}[t]
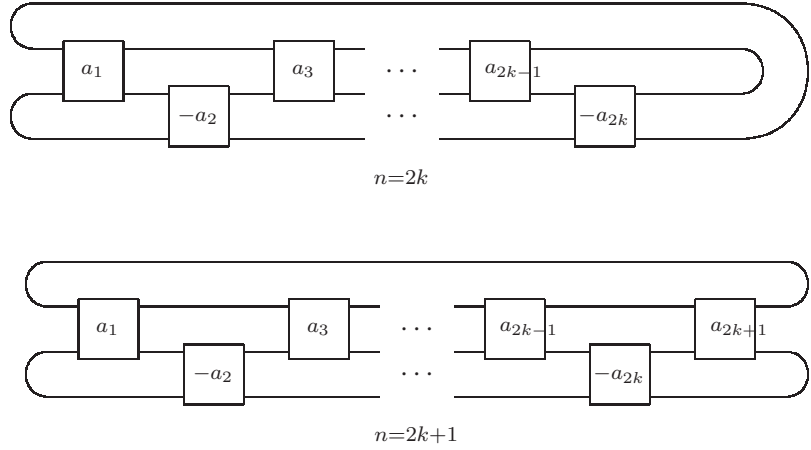

%\vspace*{10pt}
\centerline{\xy  (10,6);(18,6) **@{-},
(10,14);(18,14) **@{-}, (10,6);(10,14) **@{-}, (18,6);(18,14)
**@{-},
(24,0);(32,0) **@{-}, (24,8);(32,8) **@{-}, (24,0);(24,8) **@{-},
(32,0);(32,8) **@{-},
(38,6);(46,6) **@{-}, (38,14);(46,14) **@{-}, (38,6);(38,14)
**@{-}, (46,6);(46,14) **@{-},
(64,14);(72,14) **@{-}, (64,6);(72,6) **@{-}, (64,14);(64,6)
**@{-}, (72,14);(72,6) **@{-},
(78,8);(86,8) **@{-}, (78,0);(86,0) **@{-}, (78,8);(78,0) **@{-},
(86,8);(86,0) **@{-},
 %%%%%
(6,13);(10,13) **@{-}, (18,13);(38,13) **@{-}, (46,13);(50,13)
**@{-},
(6,7);(10,7) **@{-}, (18,7);(24,7) **@{-}, (32,7);(38,7) **@{-},
(46,7);(50,7) **@{-},
(6,1);(24,1) **@{-}, (32,1);(50,1) **@{-},
(60,13);(64,13) **@{-}, (72,13);(100,13) **@{-},
(60,7);(64,7) **@{-}, (72,7);(78,7) **@{-}, (86,7);(100,7) **@{-},
(60,1);(78,1) **@{-}, (86,1);(100,1) **@{-}, (6,19);(100,19)
**@{-},
(6,1);(6,7) **\crv{(2,1)&(2,7)}, (6,13);(6,19)
**\crv{(2,13)&(2,19)},
(100,7);(100,13) **\crv{(104,7)&(104,13)}, (100,1);(100,19)
**\crv{(112,1)&(112,19)}, (55,-4)*{_{n=2k}},
(55,4)*{\cdots}, (55,10)*{\cdots}, (14,10)*{_{a_1}},
(28,4)*{_{-a_2}}, (42,10)*{_{a_3}}, (70,10)*{_{a_{2k-1}}},
(82,4)*{_{-a_{2k}}},
\endxy }\vskip1cm
\centerline{\xy  (10,6);(18,6) **@{-}, (10,14);(18,14) **@{-},
(10,6);(10,14) **@{-}, (18,6);(18,14) **@{-},
(24,0);(32,0) **@{-}, (24,8);(32,8) **@{-}, (24,0);(24,8) **@{-},
(32,0);(32,8) **@{-},
(38,6);(46,6) **@{-}, (38,14);(46,14) **@{-}, (38,6);(38,14)
**@{-}, (46,6);(46,14) **@{-},
(64,14);(72,14) **@{-}, (64,6);(72,6) **@{-}, (64,14);(64,6)
**@{-}, (72,14);(72,6) **@{-},
(78,8);(86,8) **@{-}, (78,0);(86,0) **@{-}, (78,8);(78,0) **@{-},
(86,8);(86,0) **@{-},
(92,14);(100,14) **@{-}, (92,6);(100,6) **@{-}, (92,14);(92,6)
**@{-}, (100,14);(100,6) **@{-},
 %%%%%
(6,13);(10,13) **@{-}, (18,13);(38,13) **@{-}, (46,13);(50,13)
**@{-},
(6,7);(10,7) **@{-}, (18,7);(24,7) **@{-}, (32,7);(38,7) **@{-},
(46,7);(50,7) **@{-},
(6,1);(24,1) **@{-}, (32,1);(50,1) **@{-},
(60,13);(64,13) **@{-}, (72,13);(92,13) **@{-}, (100,13);(104,13)
**@{-},
(60,7);(64,7) **@{-}, (72,7);(78,7) **@{-}, (86,7);(92,7) **@{-},
(100,7);(104,7) **@{-},
(60,1);(78,1) **@{-}, (86,1);(104,1) **@{-}, (6,19);(104,19)
**@{-},
(6,1);(6,7) **\crv{(2,1)&(2,7)}, (6,13);(6,19)
**\crv{(2,13)&(2,19)},
(104,1);(104,7) **\crv{(108,1)&(108,7)}, (104,13);(104,19)
**\crv{(108,13)&(108,19)},  (55,-4)*{_{n=2k+1}},
(55,4)*{\cdots}, (55,10)*{\cdots}, (14,10)*{_{a_1}},
(28,4)*{_{-a_2}}, (42,10)*{_{a_3}}, (70,10)*{_{a_{2k-1}}},
(82,4)*{_{-a_{2k}}}, (98,10)*{_{a_{2k+1}}},
\endxy }
\vspace*{10pt} \caption{$C(a_1, a_2, \ldots, a_n)$}\label{fig0.1}
\end{figure}
\begin{figure}[t]
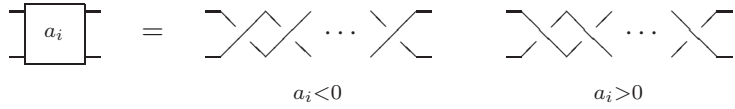

%\vspace*{10pt}
\centerline{\xy (2,0);(10,0) **@{-}, (2,8);(10,8)
**@{-}, (2,0);(2,8) **@{-}, (10,0);(10,8) **@{-}, (0,1);(2,1)
**@{-}, (0,7);(2,7) **@{-}, (10,1);(12,1) **@{-},  (10,7);(12,7)
**@{-}, (6,4)*{_{a_i}}, (19,4)*{=},
 %%%
(28,1);(26,1)  **@{-}, (28,7);(26,7)  **@{-}, (28,1);(34,7)
**@{-}, (28,7);(30,5) **@{-}, (32,3);(34,1) **@{-}, (34,1);(40,7)
**@{-}, (34,7);(36,5) **@{-}, (38,3);(40,1) **@{-},
(44,4)*{\cdots}, (48,1);(54,7) **@{-}, (48,7);(50,5)  **@{-},
(52,3);(54,1) **@{-}, (56,1);(54,1) **@{-}, (56,7);(54,7) **@{-},
 %%%
(68,7);(66,7)  **@{-}, (68,1);(66,1)  **@{-}, (68,7);(74,1)
**@{-}, (68,1);(70,3) **@{-}, (72,5);(74,7) **@{-}, (74,7);(80,1)
**@{-}, (74,1);(76,3) **@{-}, (78,5);(80,7) **@{-},
(84,4)*{\cdots}, (88,7);(94,1) **@{-}, (88,1);(90,3)  **@{-},
(92,5);(94,7) **@{-}, (96,1);(94,1) **@{-}, (96,7);(94,7) **@{-},
 %%%
 (41,-4)*{_{a_i < 0}}, (81,-4)*{_{a_i > 0}},
\endxy
 } \vspace*{10pt} \caption{Half-twists}\label{fig0.2}
\end{figure}
We close this section with the following proposition which comes from \cite[Proposition 5]{Nak} immediately.

\begin{proposition}\label{prop-Nakamura}
Let $D=C(a_1, a_2, \ldots, a_n)$ with $a_i > 0$ for $i=1, 2, \ldots, n$. Then
\[\max\deg_z P_{W_2(D)}(v,z)=2c(D)-1.\]
\end{proposition}

\section{Proof of Theorem \ref{main-thm1}}
\label{sect-1-braid}

In this section, we prove Theorem \ref{main-thm1}. For this purpose, we first prove the following:

\begin{lemma}\label{lem-max-d-lij}
Let $L_{i}$ be a $(2,i)$-torus link, the closure of the braid $\sigma_1^i \in B_2$ (with parallel orientation) as shown in Fig.~\ref{con-1}, and  let $L_{i,j}=L_{i}\sharp L_{j}$ be the connected sum of two torus links $L_i$ and $L_j$ with $i, j \geq 2$ as shown in Fig.~\ref{con-2}. Then
\begin{equation*}%\label{tor-1}
\max\deg_z P_{W_2(L_{i,j})}(v,z)=2(i+j)-3=2c(L_{i,j})-3.
\end{equation*}

\begin{figure}[ht]
\begin{center}
\resizebox{0.40\textwidth}{!}{%
\includegraphics{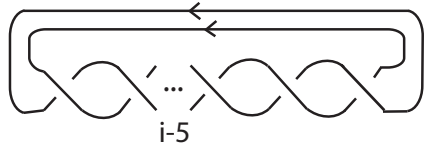} }
\caption{A $(2,i)$-torus link $L_{i}$}\label{con-1}
\end{center}
\end{figure}

\begin{figure}[ht]
\begin{center}
\resizebox{0.85\textwidth}{!}{%
  \includegraphics{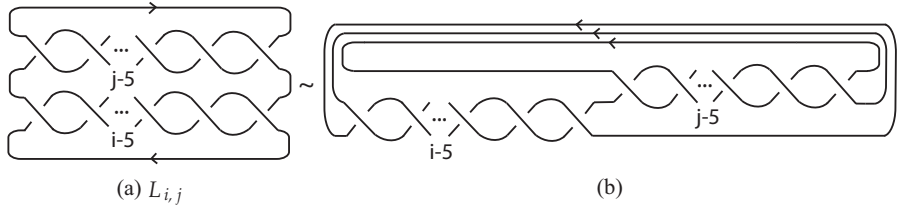} }
\caption{$L_{i,j}=L_i\sharp L_j$}\label{con-2}
\end{center}
\end{figure}
\end{lemma}

\begin{figure}[ht]
\begin{center}
\resizebox{0.80\textwidth}{!}{%
\includegraphics{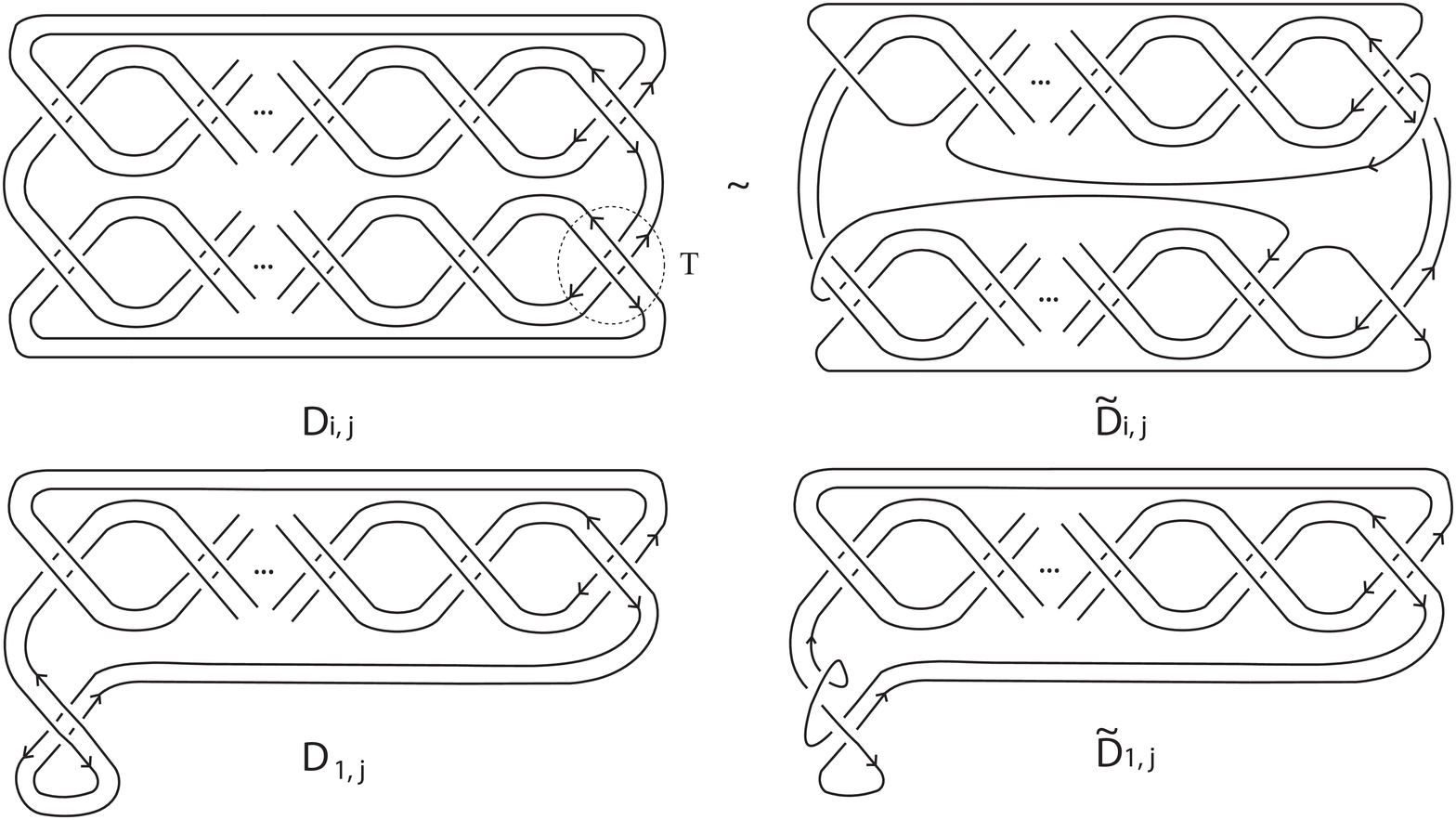} }
\caption{Two diagrams of $W_{2}(L_{i,j})$}
\label{tw1}
\end{center}
\end{figure}

\begin{proof}
For any pair $i\geq 1$ and $j\geq 2$, let $D_{i,j}$ denote the standard diagram of the doubled link $W_{2}(L_{i,j})$ of the connected sum $L_{i,j}$ as shown in the left-hand side of Fig.~\ref{tw1} and we consider another diagram $\tilde D_{i,j}$ of $W_{2}(L_{i,j})$, which is obtained from $D_{i,j}$ by isotopy deformations as illustrated in the right-hand side of Fig.~\ref{tw1}. For our convenience, for each $j\geq 2$ we define $D_{0,j}$ to be the standard diagram of the doubled link $W_{2}(L_{j})$ of a $(2,j)$-torus link $L_{j}$ and then define $\tilde D_{0,j}= D_{0,j} \amalg O^2$, the split union of $D_{0,j}$ and the $2$-component trivial link $O^2$. Then $\max\deg_z P_{W_2(L_{i,j})}(v,z)=M(D_{i,j})=M(\tilde D_{i,j})$ for $i\geq 1$ and $j\geq 2$, and $c(\tilde D_{i,j})=4(i+j)=c(D_{i,j})$ and $s(\tilde D_{i,j})=2(i+j)+4=s(D_{i,j})+2$ for $i\geq 0$ and $j\geq 2$. Note that if $i,j \geq 2$, then $L_{ij}$ is a reduced alternating diagram (see (a) in Fig.~\ref{con-2}) and so $c(L_{i,j})=i+j$. 

For $i\geq 0$ and $j\geq 2$, let $N_{i,j}$ denote the integer defined by
\begin{equation*}
N_{i,j}=c(\tilde D_{i,j})-s(\tilde D_{i,j})+1=4(i+j)-\{2(i+j)+4\}+1=2(i+j)-3.
\end{equation*}
By Morton's inequality in (\ref{Morton-ineq-c-gen}), we obtain that for any pair $i\geq 1$ and $j\geq 2$,
\begin{equation*}%\label{pf-m-ineq-1}
\max\deg_z P_{W_2(L_{i,j})}(v,z)=M(D_{i,j})=M(\tilde D_{i,j}) \leq N_{i,j}.
\end{equation*}
Indeed, what we want to prove is that the equality 
\begin{equation}\label{pf-claim-1}
M(\tilde D_{i,j}) = N_{i,j}
\end{equation} holds for any pair $i\geq 1$ and $j\geq 2$. For any given fixed integer $j \geq 2$, we prove the assertion (\ref{pf-claim-1}) by induction on $i\geq 1$. 

In \cite[Proposition 1]{Tri}, it is known that $\max\deg_z P_{W_2(L_{j})}(v,z)=2j-1$ for each integer $j\geq 2$. Since $L_{1,j}=L_1\sharp L_j=L_j$, we obtain that
\begin{align*}
\max\deg_z P_{W_2(L_{1,j})}(v,z)=M(D_{1,j})=M(\tilde D_{1,j})=2j-1=2(1+j)-3=N_{1,j}.
\end{align*} 
This gives that the assertion (\ref{pf-claim-1}) holds for the initial step $i=1$.

\bigskip
 
Now we assume that $i \geq 2$ and the assertion (\ref{pf-claim-1}) holds for every integers $k$ with $1\leq k \leq i-1$.
We consider a partial skein tree for the tangle $T$ in $D_{i,j}$ as shown in Fig.~\ref{tangle2}. We label all nodes in the partial skein tree with $A,~B,~E_{1},~F_{1},~F_{2},~F_{3},~F_{4}$ and $G$ as in Fig.~\ref{tangle2}.

\begin{figure}[ht]
\begin{center}
\resizebox{0.35\textwidth}{!}{%
\includegraphics{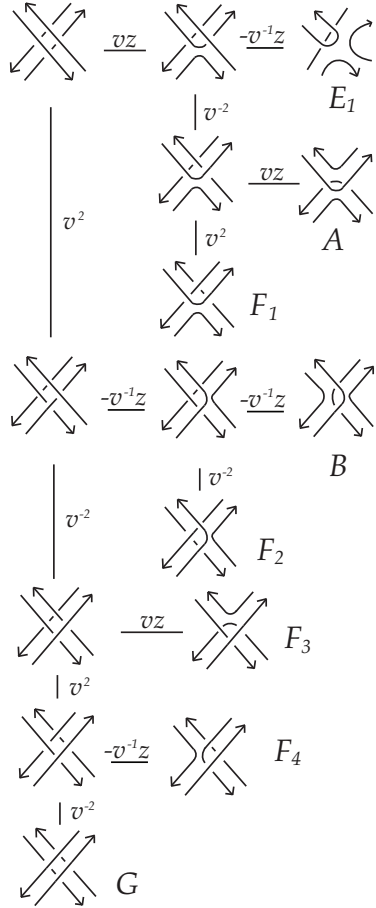} }
\caption{A partial skein tree for $T$.}
\label{tangle2}
\end{center}
\end{figure}

For each $k=1,2,\ldots, 8$, let $D_{i,j}^{k}$ be the link diagram obtained from $D_{i,j}$ by replacing the tangle $T$ with the tangle $T_{k}$, where
\begin{equation}\label{tang-assign}
T_{1}=A,~ T_{2}=B,~T_{3}=E_{1},~T_{4}=F_{1},~T_{5}=F_{2},~ T_{6}=F_{3},~ T_{7}=F_{4},~ T_{8}=G.\end{equation}
Note that two diagrams $D_{i,j}$ and $D_{i,j}^{k}$ are identical except the parts of them corresponding to the tangle $T$.
From the skein relation for the HOMFLYPT polynomial and a partial skein tree for the tangle $T$ in $D_{i,j}$, we obtain
\begin{align}\label{tw-q-1}
&P_{D_{i,j}}(v,z)=(P_{D_{i,j}^{1}}(v,z)+ P_{D_{i,j}^{2}}(v,z)-P_{D_{i,j}^{3}}(v,z))z^{2}+(vP_{D_{i,j}^{4}}(v,z)\notag\\
&~~~~~-v^{-1}P_{D_{i,j}^{5}}(v,z)+vP_{D_{i,j}^{6}}(v,z)-vP_{D_{i,j}^{7}}(v,z))z+P_{D_{i,j}^{8}}(v,z).
\end{align}
Using this equation, we are going to calculate the maximum degree in $z$ of $P_{D_{i,j}}(v,z)$ ($=P_{\tilde D_{i,j}}(v,z)$).
We first observe that $D^1_{i,j}$ and $D^8_{i,j}$ are isotopic to $\tilde D_{i-1,j}$ and $\tilde D_{i-2,j}$, respectively. Hence it follows from induction hypothesis that
\begin{align}
&M(\tilde D_{i-1,j})=N_{i-1,j}~ (i\geq 2),\notag\\
&M(\tilde D_{i-2,j})=N_{i-2,j}~ (i\geq 3).\label{pf-ind-st}
\end{align}
For $i=2$ in (\ref{pf-ind-st}), it is easily seen that
\begin{align}
M(\tilde D_{0,j})&=M(D_{0,j} \amalg O^2)=M(W_2(L_j) \amalg O^2)\notag\\&=M(W_2(L_j))-2=2j-3=N_{0,j}.\label{pf-d02-eq1}
\end{align}
Hence we have
\begin{align}
&\max\deg_{z}P_{D_{i,j}^{1}}(v,z)=M(\tilde D_{i-1,j})=N_{i-1,j}=N_{i,j}-2~ (i\geq 2),\label{tw-q-2}\\
&\max\deg_{z}P_{D_{i,j}^{8}}(v,z)=M(\tilde D_{i-2,j})=N_{i-2,j}=N_{i,j}-4~ (i\geq 2).\label{tw-qq-2}
\end{align}
It is evident that the link $D_{i,j}^{2}$ and $D_{i,j}^{3}$ do not contribute anything to $M(D_{i,j})=\max\deg_z P_{D_{i,j}}(v,z)$. 

\begin{figure}[ht]
\begin{center}
\resizebox{0.70\textwidth}{!}{%
  \includegraphics{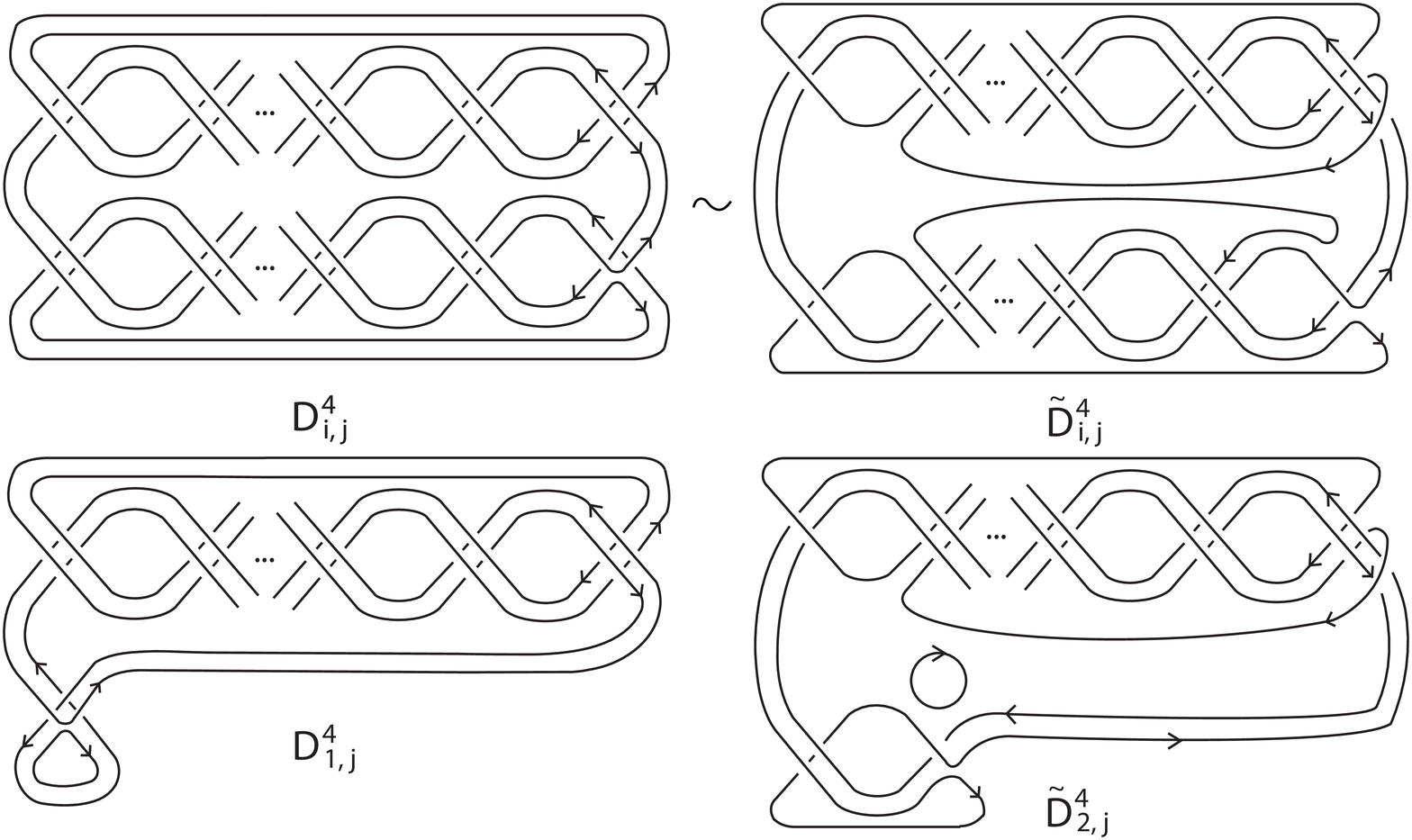} }
\caption{Two diagrams of $D_{i,j}^{4}$}
\label{tw2}
\end{center}
\end{figure}

\begin{figure}[h]
\begin{center}
\resizebox{0.75\textwidth}{!}{%
\includegraphics{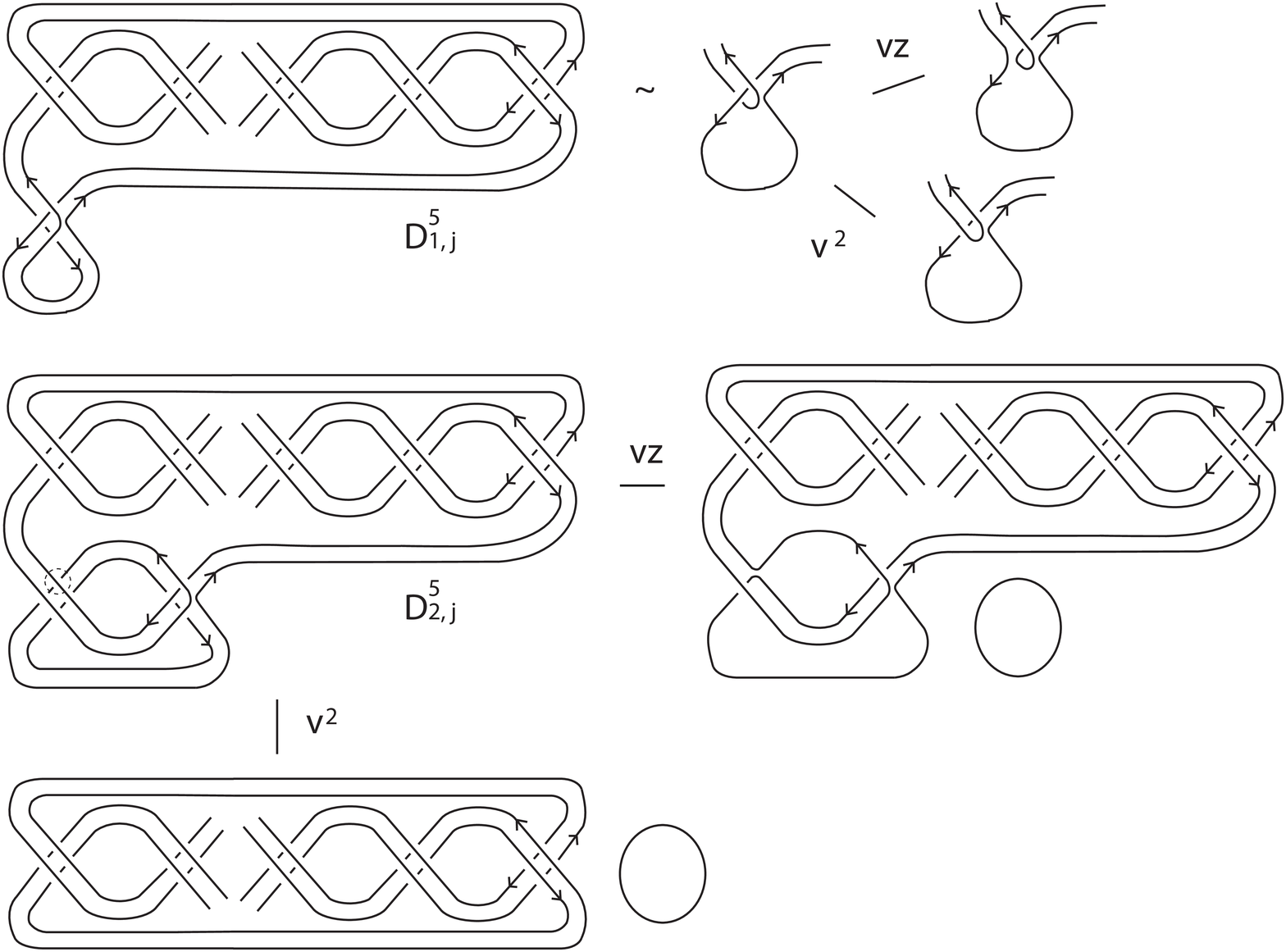} }
\caption{Partial skein trees for $D_{1,j}^{5}$ and $D_{2,j}^{5}$}
\label{tw5-2}
\end{center}
\end{figure}

To estimate $M(D_{i,j}^{4})$, we consider a link diagram $\tilde D_{i,j}^{4}$ obtained from $D_{i,j}^{4}$ by isotopy deformations as illustrated in Fig.~\ref{tw2}. Then it follows that
\begin{align}
\max\deg_{z}P_{D_{i,j}^{4}}(v,z)
&=\max\deg_{z}P_{{\tilde D}_{i,j}^{4}}(v,z)
\leq c({\tilde D}_{i,j}^{4})-s({\tilde D}_{i,j}^{4})+1\notag\\
&= (c({\tilde D}_{i,j})-5)-(s({\tilde D}_{i,j})-2)+1=N_{i,j}-3.
\label{tw-q-3}
\end{align}

\begin{figure}%[t]
\begin{center}
\resizebox{0.65\textwidth}{!}{%
  \includegraphics{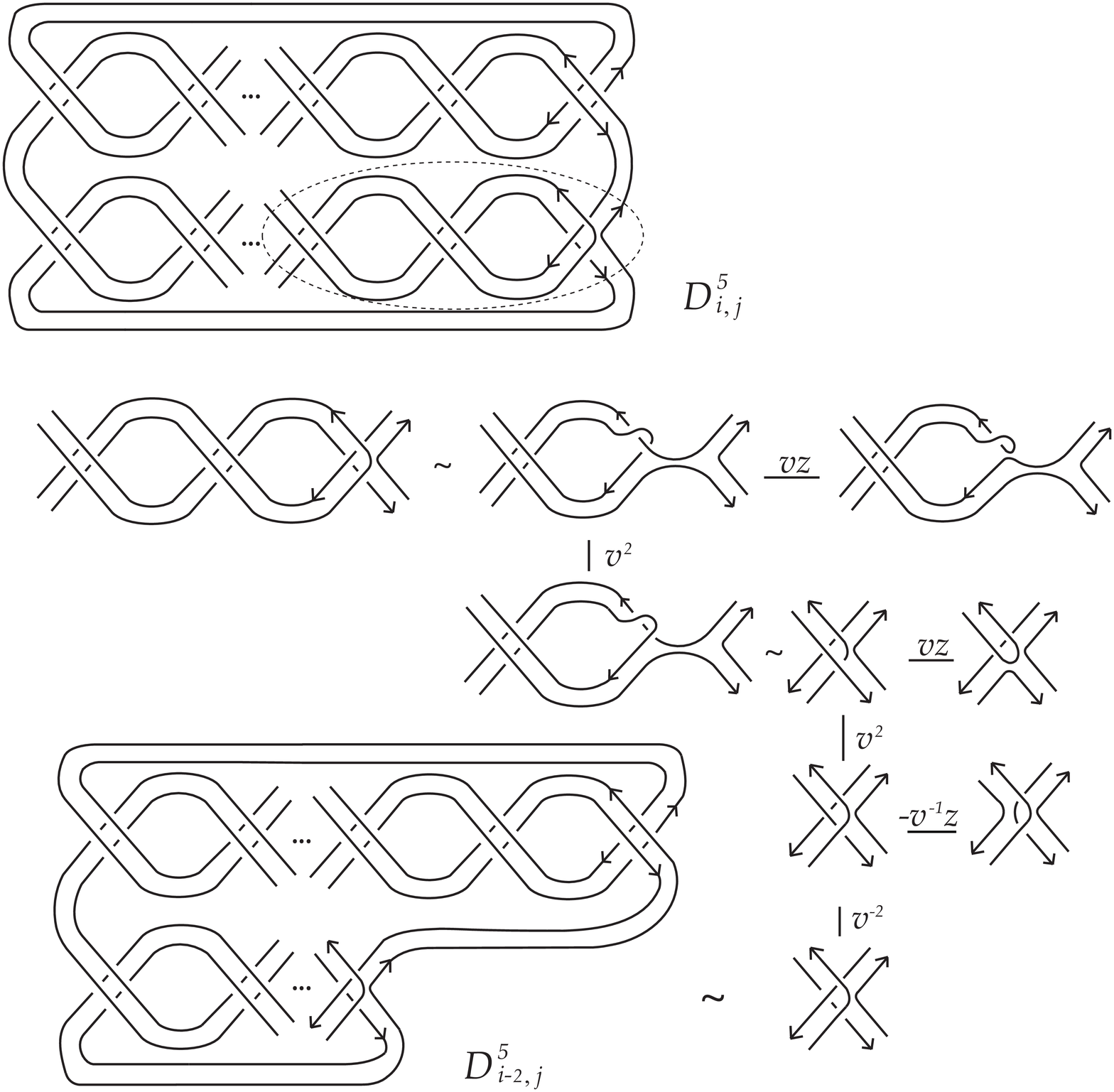} }
\caption{A partial skein tree for $D_{i,j}^{5} (i\geq 3)$}
\label{tw3}
\end{center}
\end{figure}

For $D_{i,j}^{5}$, if $i=1$, then we observe from Fig.~\ref{tw5-2} that 
\begin{equation}
M(D_{1,j}^{5})=M(W_2(L_j))+1=2j. \label{pf-5i-2-1}
\end{equation}
If $i=2$, then we observe from Fig.~\ref{tw5-2} that 
\begin{align}
M(D_{2,j}^{5})&=M(W_2(L_j)\amalg O)=M(W_2(L_j))-1=2j-2\leq N_{2,j}-3. \label{pf-5i-2-2}
\end{align}
If $i\geq 3$, then the partial skein trees in Fig.~\ref{tw3} yield $$\max\deg_{z}P_{D_{i,j}^{5}}(v,z)=M(D_{i-2,j}^{5}).$$ Hence 
\begin{align}
M(D_{i,j}^{5})&= 
\begin{cases}
M(D^5_{1,j})~\text{if $i$ is odd $\geq 3$;}\\
M(D^5_{2,j})~\text{if $i$ is even $\geq 4$.}
\end{cases}=\begin{cases}
2j~\text{if $i$ is odd $\geq 3$;}\\
2j-2~\text{if $i$ is even $\geq 4$.}
\end{cases}\notag\\
&\leq N_{i,j}-3~(i\geq 3).\label{pf-5i-2-3}
\end{align}
Thus we obtain from (\ref{pf-5i-2-2}) and (\ref{pf-5i-2-3}) that
\begin{align}
&\max\deg_{z}P_{D_{i,j}^{5}}(v,z)\leq N_{i,j}-3.\label{tw-q-4}
\end{align}

\begin{figure}[ht]
\begin{center}
\resizebox{0.50\textwidth}{!}{%
  \includegraphics{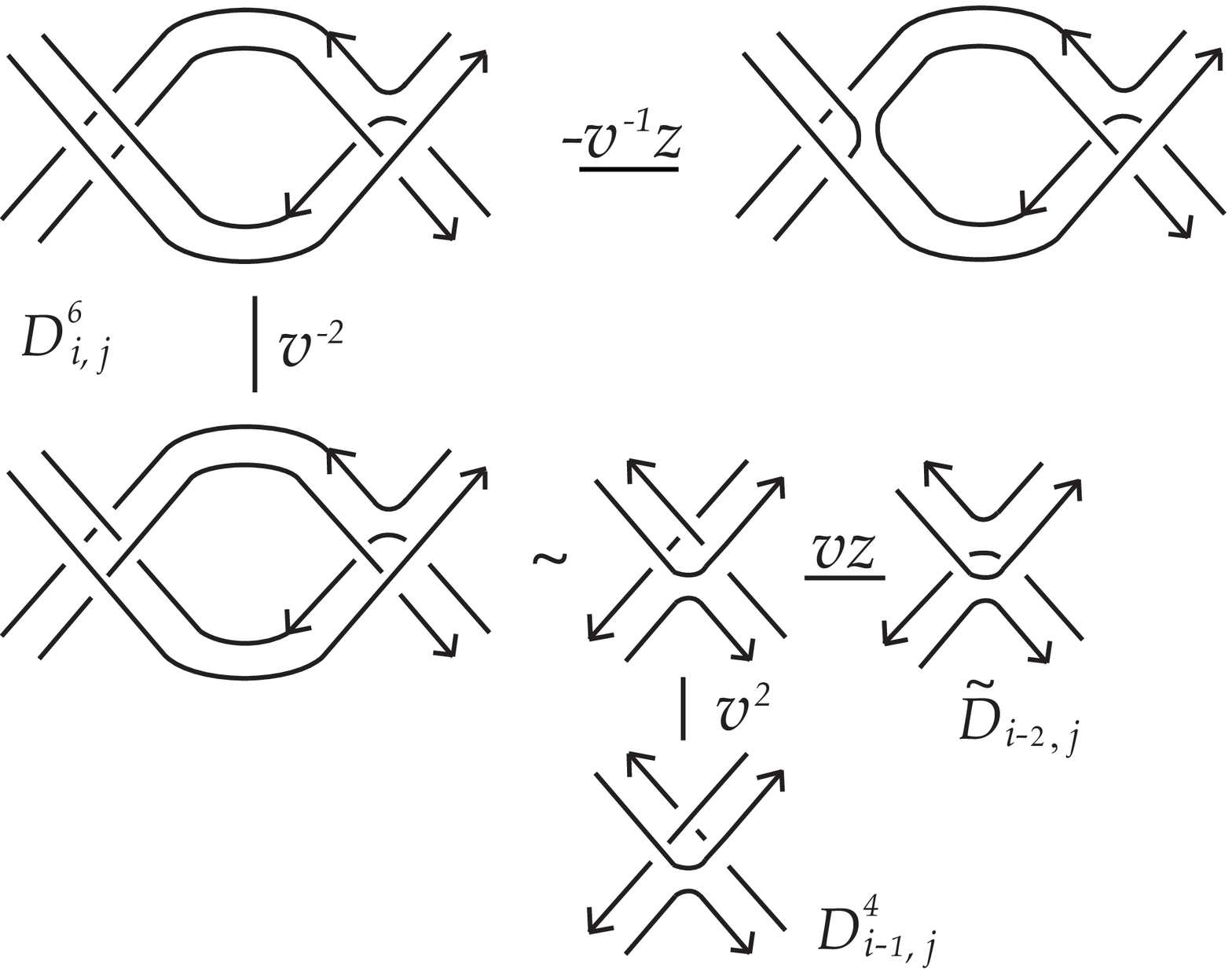} }
\caption{A partial skein tree for $D_{i,j}^{6}$}
\label{tw4}
\end{center}
\end{figure}

For $D_{i,j}^{6}$, the partial skein trees in Fig.~\ref{tw4} yield
\begin{align*}
&\max\deg_{z}P_{D_{i,j}^{6}}(v,z)\leq \max\{M(D_{i-1,j}^{4}),~M(\tilde D_{i-2,j})+1\}.%
\end{align*}
We remind that $M(\tilde D_{i-2,j})=N_{i-2,j}=N_{i,j}-4$ shown in (\ref{pf-ind-st}) and (\ref{pf-d02-eq1}). Observe that $M(D_{1,j}^{4})=M(W_2(L_j)\amalg O)=M(W_2(L_j))-1=2j-2=N_{2,j}-3$ (see Fig.~\ref{tw2}).
And, if $i\geq 3$, then it follows from the Morton's inequality in (\ref{Morton-ineq-c-gen}) that
\begin{align*}
M(D_{i-1,j}^{4})&=M(\tilde D_{i-1,j}^{4}) \leq c(\tilde D_{i-1,j}^{4})-s(\tilde D_{i-1,j}^{4})+1\\%\notag\\
&= (c(\tilde D_{i,j})-9)-(s(\tilde D_{i,j})-4)+1=N_{i,j}-5.%\notag
\end{align*}
These observations gives
\begin{align}
&\max\deg_{z}P_{D_{i,j}^{6}}(v,z)\leq \max\{N_{i,j}-5,~N_{i,j}-3\}=N_{i,j}-3. \label{tw-q-5}
\end{align}

Now we estimate the maximum degree in $z$ of $P_{D_{i,j}^{7}}(v,z)$. Observe that $D_{i,j}^{7}$ is clearly isotopic to the diagram $D_1$ in Fig.~\ref{tw5}. For $i=2,$ it is easy to see that $D_{2,j}^{7}=W_2(L_j)\amalg O$ and so $M(D_{2,j}^{7})=N_{2,j}-3$ as seen in (\ref{pf-5i-2-1}).
For $i\geq 3$, moving two crossings of $D_1$ labeled $1, 2$ along $2$-parallel strings by isotopy, they appear in the place adjacent to the crossing labeled $3, 4,$ respectively, as indicated in $D_1$ or $D_2$ according to the parity of $i$, and two parallel strings of the components in $D_{1}$ under consideration are switched each other. Hence the resulting diagram after applying Reidemeister move of type II yield the diagram $D_3$ in Fig.~\ref{tw5} with reverse orientations on the components in $D_{1}$ under consideration. Obviously, we can reverse orientations of the remaining components in $D_3$ (if they exist) by isotopy. From the partial skein tree for $D_3$ in Fig.~\ref{tw6} together with (\ref{pf-5i-2-1}) and (\ref{tw-q-4}), we obtain \begin{align*}
\max\deg_{z}P_{D_{i,j}^{7}}(v,z)&=\max\deg_{z}P_{D_3}(v,z)=\max\deg_{z}P_{-D_{i-2,j}^{5}}(v,z)\notag\\
&=\max\deg_{z}P_{D_{i-2,j}^{5}}(v,z) \leq N_{i,j}-3,
\end{align*} where $-D_{i-2,j}^{5}$ is the diagram $D_{i-2,j}^{5}$ with reversed orientation as shown in Fig.~\ref{tw7} (cf. Fig.~\ref{tw3}). These observations implies
\begin{equation}
\max\deg_{z}P_{D_{i,j}^{7}}(v,z) \leq N_{i,j}-3, \label{tw-q-6}
\end{equation} 

\begin{figure}[ht]
\begin{center}
\resizebox{0.70\textwidth}{!}{%
  \includegraphics{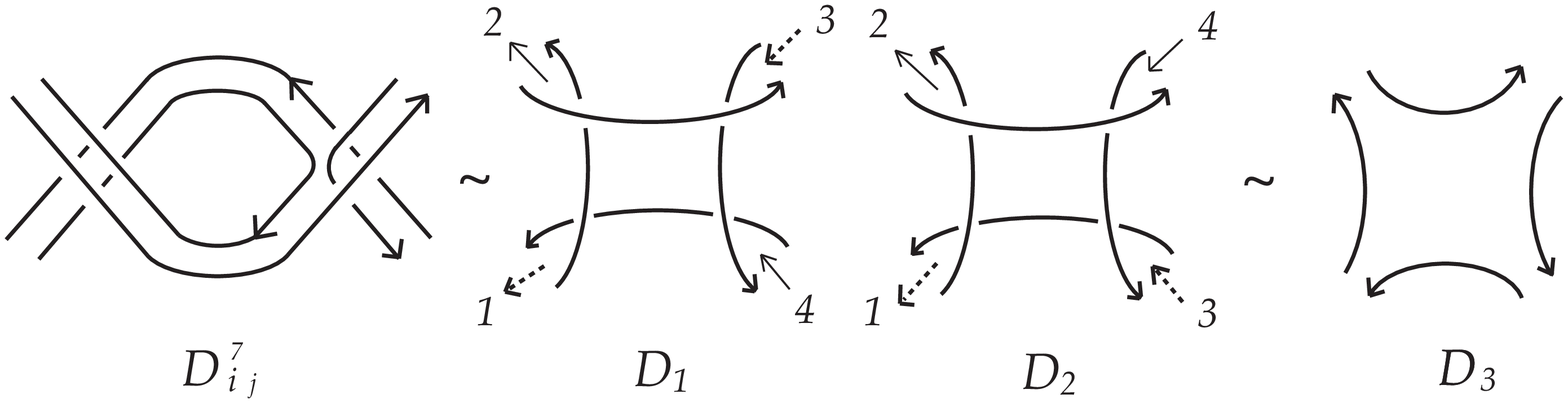} }
\caption{$D_{i,j}^{7}$}\label{tw5}
\end{center}
\end{figure}

\begin{figure}[ht]
\begin{center}
\resizebox{0.55\textwidth}{!}{%
  \includegraphics{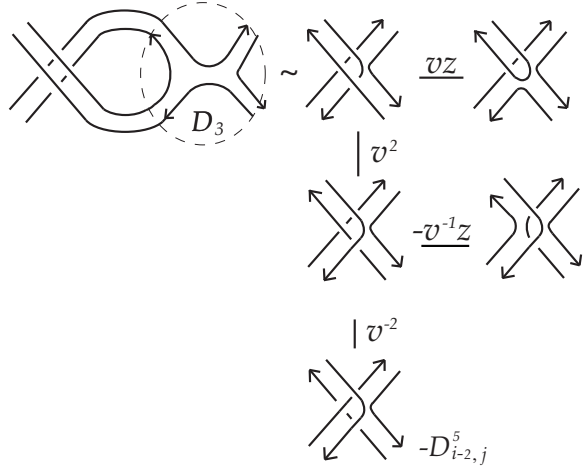} }
\caption{A partial skein tree for $D_{3}$}\label{tw6}
\end{center}
\end{figure}

\begin{figure}[ht]
\begin{center}
\resizebox{0.45\textwidth}{!}{%
  \includegraphics{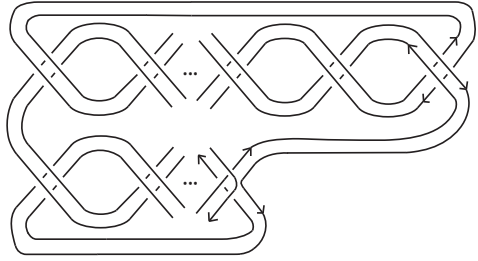} }
\caption{$-D_{i-2,j}^{5}$}\label{tw7}
\end{center}
\end{figure}

Combining (\ref{tw-q-1}), (\ref{tw-q-2})-(\ref{tw-q-3}) (\ref{tw-q-4}) and (\ref{tw-q-4})-(\ref{tw-q-6}), we obtain that 
\begin{align*}
\max&\deg_{z}P_{D_{i,j}}(v,z) = \max\{M(D_{i,j}^{1})+2, M(D_{i,j}^{2})+2,M(D_{i,j}^{3})+2,\notag\\
 & M(D_{i,j}^{4})+1, M(D_{i,j}^{5})+1,M(D_{i,j}^{6})+1, M(D_{i,j}^{7})+1, M(D_{i,j}^{8})\}\notag\\
  &=N_{i,j}=2(i+j)-3~\text{for all}~i\geq 1.
\end{align*}
This establishes the equality (\ref{pf-claim-1}). Finally, if $i, j \geq 2$, then $i+j=c(L_{i,j})$. This completes the proof of Lemma \ref{lem-max-d-lij}.
\end{proof}

\begin{proof}[{\bf Proof of Theorem \ref{main-thm1}.}]
Let $i, j$ be given odd integers $\geq 3$, let $K_i$ and $K_j$ denote the $(2, i)$- and $(2,j)$-torus knot, respectively, and let $K_{i,j}$ be the connected sum of $K_i$ and $K_j$, i.e., $K_{i,j}=K_i \sharp K_j.$ Then it follows from Lemma \ref{lem-max-d-lij} that
\begin{equation}\label{eq1-pf-thm1}
\max\deg_z P_{W_2(K_{i,j})}(v,z)=2(i+j)-3=2c(K_{i,j})-3.
\end{equation}
For any given integer $m$, let $W_+(K_{i,j},m)$ be the $m$-twisted positive Whitehead double of $K_{i,j}$ and let $W_+(L_{i,j},m)$ be the canonical diagram of $W_\pm(K_{i,j},m)$ associated with the diagram $L_{i,j}$ in Fig.~\ref{con-2}. Since $c(K_{i,j}) \geq 6$, it follows from (\ref{eq1-pf-thm1}) and Proposition \ref{prop3-cr-nbr-cg-wd} that $\max\deg_z P_{W_+(K_{i,j},m)}(v,z) > 0$ and hence\\ $\max\deg_z P_{W_2(L_{i,j}, w(L_{i,j}))}(v,z) \not= 1$. By Proposition \ref{prop3-cr-nbr-cg-wd}, we have
\begin{align}
 \max\deg_z P_{W_+(K_{i,j},m)}(v,z)
 &=\max\deg_z P_{W_+(L_{i,j},m)}(v,z)\notag\\
 &=\max\deg_z P_{W_2(L_{i,j},m)}(v,z) +1\notag\\
 &=\max\deg_z P_{W_2(L_{i,j},w(L_{i,j}))}(v,z) +1\notag\\
 &=\max\deg_z P_{D_{i,j}}(v,z)+1\notag\\
 &=2c(K_{i,j})-3+1=2c(K_{i,j})-2.\label{eq2-pf-thm1}
\end{align}

\begin{figure}[ht]
\begin{center}
\resizebox{0.45\textwidth}{!}{%
  \includegraphics{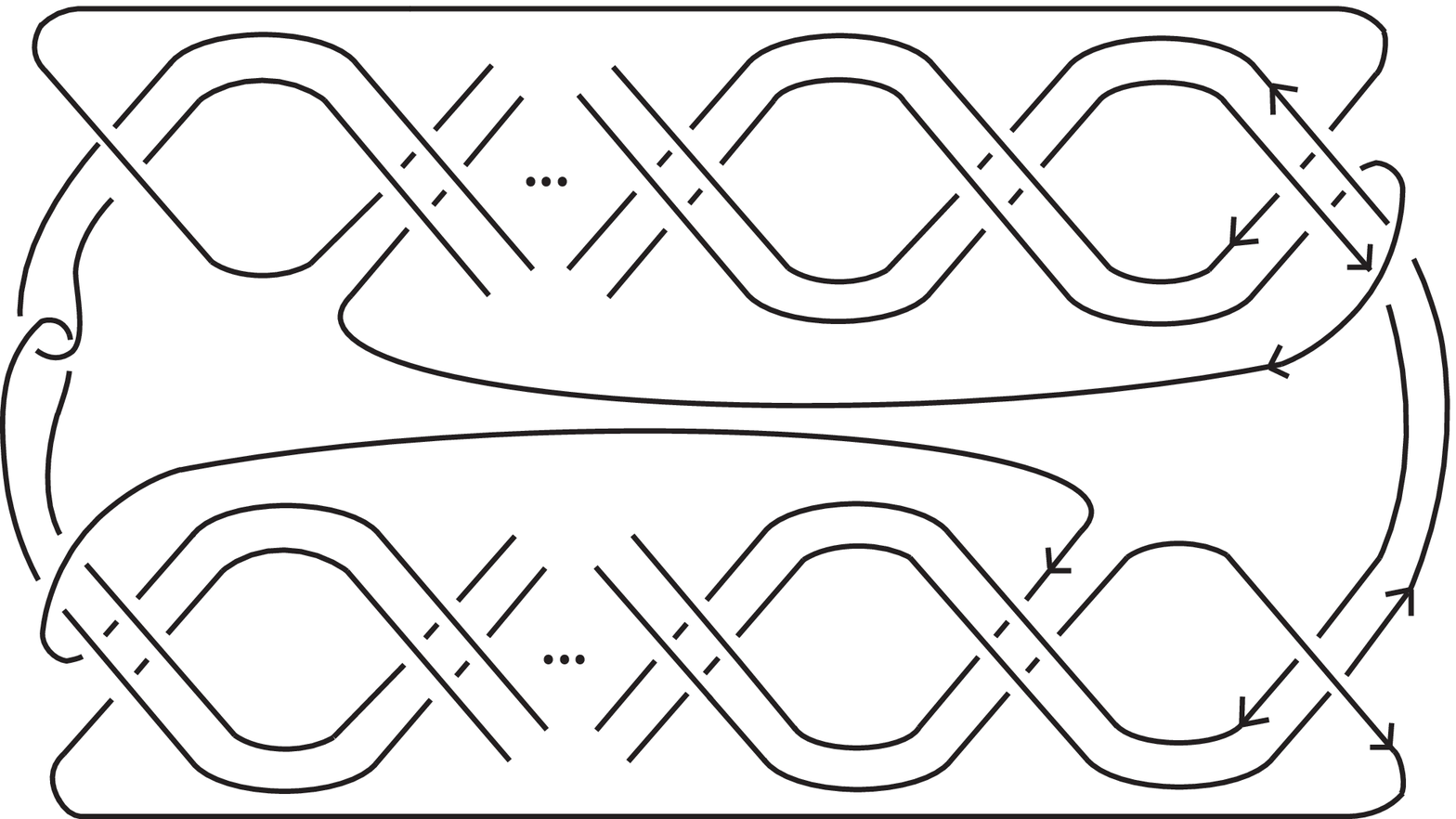} }
\caption{$D'$}\label{con-2-1}
\end{center}
\end{figure}

Now we deform the diagram $W_{+}(L_{i,j},w(L_{i,j}))$ to the diagram $D'$ as shown in Fig.~\ref{con-2-1} by using isotopy. So $g_c(W_+(K_{i,j},m)) \leq g(\Sigma(D'))$.
Observe that there are $2(i+j)+5$ Seifert circles in $D'$ that result from applying Seifert's algorithm to the diagram $D'$. Since $D'$ has $4(i+j)+2$ crossings, the genus $g(\Sigma(D'))$ of the resulting canonical Seifert surface $\Sigma(D')$ is given by
\begin{align}
g(\Sigma(D'))&=\frac{c(D')-s(D')+1}{2}=\frac{4(i+j)+2-(2(i+j)+5)+1}{2}\notag\\
&=i+j-1=c(K_{i,j})-1.\label{eq3-pf-thm1}
\end{align}

Finally, it follows from Proposition \ref{prop2-cr-nbr-cg-wd}, (\ref{eq2-pf-thm1}) and (\ref{eq3-pf-thm1}) that
\begin{align*}
c(K_{i,j})-1&=\frac{1}{2} \max\deg_z P_{W_+(K_{i,j},m)}(v,z) \leq g_c(W_+(K_{i,j},m))\notag\\ &\leq g(\Sigma(D'))=i+j-1=c(K_{i,j})-1.
\end{align*}
This gives $g_c(W_+(K_{i,j},m))=i+j-1=c(K_{i,j})-1$. By the same argument, we obtain $g_c(W_-(K_{i,j},m))=i+j-1=c(K_{i,j})-1$.
This completes the proof of Theorem \ref{main-thm1}.
\end{proof}

\begin{figure}%[h]
\begin{center}
\resizebox{0.65\textwidth}{!}{%
  \includegraphics{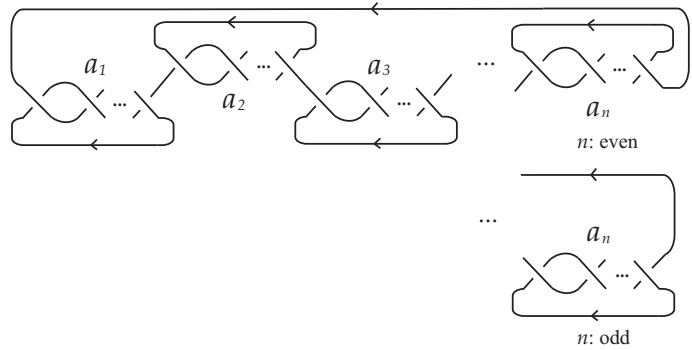} }
\caption{$D_n= D_{a_1,\ldots,a_n}$}\label{g-con-1}
\end{center}
\end{figure}

\begin{figure}%[h]
\begin{center}
\resizebox{0.75\textwidth}{!}{%
  \includegraphics{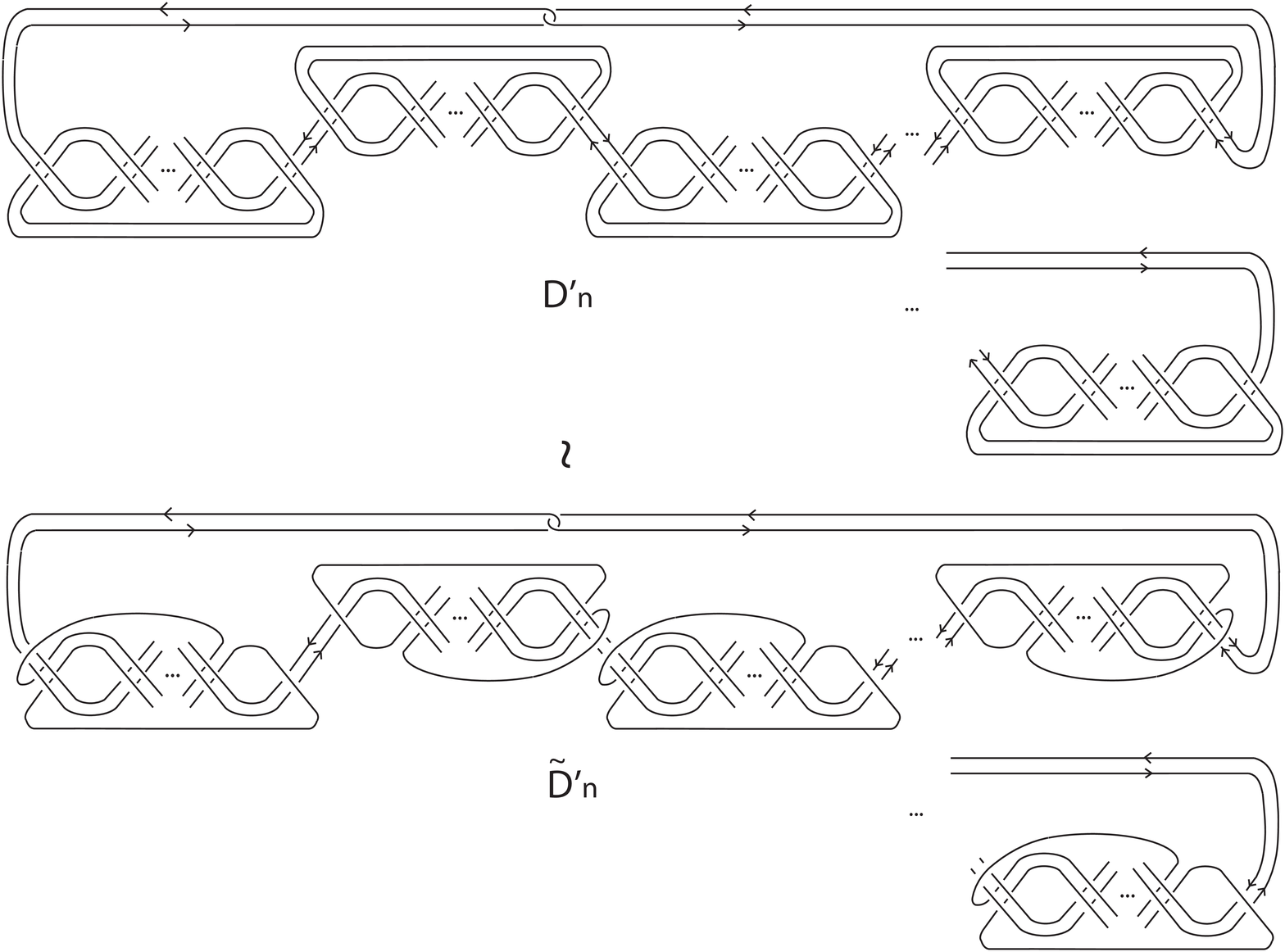} }
\caption{Two diagrams of $W_2(K_n)$}\label{g-con-2}
\end{center}
\end{figure}

\begin{remark}\label{conj-rmk-1}
(1) By a direct calculation, $\max\deg_z P_{W_2(L_{2,2})}(v,z)=2c(L_{2,2})-3=5,$ $\max\deg_z P_{W_2(L_{2,3})}(v,z)=2c(L_{2,3})-3=7,$ and $\max\deg_z P_{W_2(L_{3,3})}(v,z)=2c(L_{3,3})-3=9$.

(2) Let $a_1,\ldots,a_n (n\geq 2)$ be odd integers $\geq 3$ and let $K_{a_i} (1\leq i \leq n)$ be an oriented $(2, a_i)$-torus knot. Let $K_n$ denote an oriented alternating knot represented by $D_n= D_{a_1,\ldots,a_n}$ as shown in Fig.~{\ref{g-con-1}}, which is a diagram of the connected sum of $K_{a_1}, \ldots, K_{a_n}$. Let $D'_n=W_+(D_n, w(D_n))$ be the standard diagram of the $w(D_n)$-twisted positive Whitehead double of $K_n$ associated with $D_n$ as shown in the top of Fig.~\ref{g-con-2}, where $w(D_n)=a_1+ \cdots+a_n$, the writhe of $D_n$. Consider a diagram $\tilde D'_n$ obtained from $D'_n$ by isotopy deformations as illustrated in the bottom of Fig.~\ref{g-con-2}.
Then $\tilde D'_n$ have $2\sum_{k=1}^{n}a_{k}+2n+1$ Seifert circles and $4\sum_{k=1}^{n}a_{k}+2$ crossings and so the genus $g(\Sigma(\tilde D'_n))$ of the canonical Seifert surface $\Sigma(\tilde D'_n)$ associated to $\tilde D'_n$ is given by
\begin{align*}
g(\Sigma(\tilde D'_n))&=\frac{c(\tilde D'_n)-s(\tilde D'_n)+1}{2}\\
&=\frac{1}{2}\{4\sum_{k=1}^{n}a_{k}+2-(2\sum_{k=1}^{n}a_{k}+2n+1)+1\}\\
&=\sum_{k=1}^{n}a_{k}-(n-1)=c(K_n)-(n-1).
\end{align*}
Hence for any integer $m$, $g_c(W_+(K_n, m)) \leq g(\Sigma(\tilde D'_n))$ $=c(K_n)-(n-1)$. Therefore, Conjecture \ref{Nakam-conj-0} does not hold for any alternating knot which is obtained from the connected sum of a finite number of $(2, a_i)$-torus knots $K_{a_1}, \ldots, K_{a_n}$, where $a_i (1\leq i \leq n)$ is odd integers $\geq 3$ and $n \geq 2$.
\end{remark}

%%%%%%%%%%%%%%%%%%%%%

\section{Maximum $z$-degree of HOMFLYPT polynomials of doubled links of $\hat\gamma_p$}
\label{sect-miwd3-plrq}

In this section, we calculate the maximum degree in $z$ of the HOMFLYPT polynomials of the doubled links of alternating links obtained from alternating $3$-braid links $\hat\gamma_p (p \geq 2)$ with the orientation as shown in Fig.~\ref{fig-gamma-p} by repeatedly replacing a crossing with a full twist, where $\gamma_p$ is a $3$-braid of the form:
\begin{equation}\label{braid-beta}
\gamma_p=(\sigma_2^{\epsilon}\sigma_{1}^{-\epsilon})^p,~\text{where}~ \epsilon=\pm 1.
\end{equation}

\begin{remark}\label{prop-r-qtb}

(i) $\hat\gamma_2$ is the figure eight knot (see Fig.~\ref{fig-eight}).

(ii) $\hat\gamma_p (p\geq 2)$ is a non-split alternating link without nugatory crossings and so is a minimal crossing diagram. Hence it follows that the minimal crossing number $c(\hat\gamma_p)$ of $\hat\gamma_p$ is given by $
c(\hat\gamma_p)=2p.$

(iii) If $p=3k$ for some integer $k \geq 1$, then the closed braid $\hat\gamma_p$ is an oriented link of three components, otherwise it is always an oriented knot.

(iv) For each integer $p \geq 2$, $\gamma_p$ is a quasitoric braid of type $(3,p)$ \cite{Man}.
\end{remark}

\begin{figure}[h]
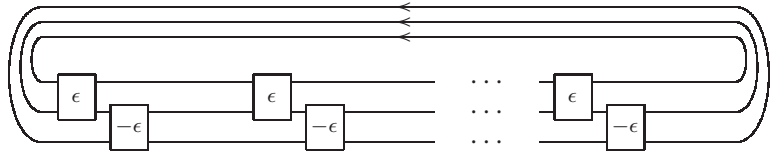

\vspace*{10pt}
\centerline{\xy
 (75,16);(80,16) **@{-}, (75,22);(80,22) **@{-},
 (75,16);(75,22) **@{-}, (80,16);(80,22) **@{-},
 (77.5,19)*{_{-\epsilon}},
 (68,20);(73,20) **@{-}, (68,26);(73,26) **@{-},
 (68,20);(68,26) **@{-}, (73,20);(73,26) **@{-},
 (70.5,23)*{_{\epsilon}},
 %%%
 (66,17);(75,17) **@{-}, (80,17);(82,17) **@{-},
 (82,17);(92,17) **@{-},
 (66,21);(68,21) **@{-}, (73,21);(75,21) **@{-},
 (80,21);(92,21) **@{-},
 (66,25);(68,25) **@{-}, (73,25);(92,25) **@{-},
 %%%
 (56,17);(61,17) **@{-},
 (49,16);(54,16) **@{-}, (49,22);(54,22) **@{-},
 (49,16);(49,22) **@{-}, (54,16);(54,22) **@{-},
 (51.5,19)*{_{-\epsilon}},
 (42,20);(47,20) **@{-}, (42,26);(47,26) **@{-},
 (42,20);(42,26) **@{-}, (47,20);(47,26) **@{-},
 (44.5,23)*{_{\epsilon}},
 %%%
 (40,17);(49,17) **@{-}, (54,17);(56,17) **@{-},
 (61,17);(66,17) **@{-},
 (40,21);(42,21) **@{-}, (47,21);(49,21) **@{-},
 (54,21);(66,21) **@{-},
 (40,25);(42,25) **@{-}, (47,25);(66,25) **@{-},
 %%%
 (122,17);(127,17) **@{-},
 (115,16);(120,16) **@{-}, (115,22);(120,22) **@{-},
 (115,16);(115,22) **@{-}, (120,16);(120,22) **@{-},
 (117.5,19)*{_{-\epsilon}},
 (108,20);(113,20) **@{-}, (108,26);(113,26) **@{-},
 (108,20);(108,26) **@{-}, (113,20);(113,26) **@{-},
 (110.5,23)*{_{\epsilon}},
 %%%
 (106,17);(115,17) **@{-}, (120,17);(122,17) **@{-},
 (127,17);(132,17) **@{-},
 (106,21);(108,21) **@{-}, (113,21);(115,21) **@{-},
 (120,21);(132,21) **@{-},
 (106,25);(108,25) **@{-}, (113,25);(132,25) **@{-},
 %%%
 (40,31);(132,31) **@{-}, (40,33);(132,33) **@{-},
 (40,35);(132,35) **@{-},
 (99,21)*{\cdots}, (99,25)*{\cdots},
 (99,17)*{\cdots},
 %%%
 (40,25);(40,31) **\crv{(38,25)&(38,31)},
 (40,21);(40,33) **\crv{(36,21)&(36,33)},
 (40,17);(40,35) **\crv{(34,17)&(34,35)},
 (132,25);(132,31) **\crv{(134,25)&(134,31)},
 (132,21);(132,33) **\crv{(136,21)&(136,33)},
 (132,17);(132,35) **\crv{(138,17)&(138,35)},
 %%%
 (87,31);(88,31) **@{} ?<*\dir{<},
 (87,35);(88,35) **@{} ?<*\dir{<},
 (87,33);(88,33) **@{} ?<*\dir{<},
 \endxy}
\caption{Closed alternating $3$-braid $\hat\gamma_{p}$}\label{fig-gamma-p}
\end{figure}
For a given oriented knot or link diagram $D$, let $W_2(D)$ denote the doubled link represented by the oriented link diagram obtained from $D$ as follows: Draw a parallel copy of $D$ pushed off $D$ to the left with respect to the orientation of $D$, and then orient the parallel copy in the opposite direction. Notice that if $D$ is a knot diagram, then $W_2(D)=W_2(D,w(D))$ described in the section 2, and if $D=D_1 \cup \cdots \cup D_m$ is a link diagram with $m$ components $D_1, \ldots, D_m$, then $W_2(D)=W_2(D_1 \cup \cdots \cup D_m)=W_2(D_1,w(D_1))\cup \cdots \cup W_2(D_m,w(D_m))$.

Now we consider the doubled link $W_2(\hat\gamma_p)$ of the alternating $3$-braid link $\hat\gamma_p$.
Notice that the link $W_2(\hat\gamma_p)$ has no full-twists of two parallel strands and each crossing of the closed braid diagram $\hat\gamma_p$ in Fig. \ref{fig-gamma-p} produces a tangle $T^{-\epsilon}_{i,j}$ as in Fig.~\ref{fig-gamma-p-db} in the standard diagram of $W_2(\hat\gamma_p)$ associated with $\hat\gamma_p$ according as $\epsilon=1$ or $\epsilon=-1$. The standard diagram of $W_2(\hat\gamma_p)$ is equivalent to the diagram shown in Fig.~\ref{tangle}.

\begin{figure}[t]
\vspace*{10pt}
\centerline{\xy
(0,0);(-18,20) **@{-} ?>*\dir{>}, (-6,0);(-24,20) **@{-} ?<*\dir{<},
(0,20);(-8,11) **@{-} ?<*\dir{<}, (-13,5.5);(-18,0) **@{-},
(-6,20);(-11,14.5) **@{-},  (-16,9);(-24,0) **@{-}  ?>*\dir{>},
(-12.8,12.4);(-14.3,10.8) **@{-},
(-11.3,7.3);(-9.5,9.3) **@{-},
(-11.5,-5) *{\epsilon=-1},
\endxy
\qquad\qquad\qquad
\xy
(0,0);(18,20) **@{-} ?<*\dir{<}, (6,0);(24,20) **@{-} ?>*\dir{>},
(0,20);(8,11) **@{-}, (13,5.5);(18,0) **@{-} ?>*\dir{>},
(6,20);(11,14.5) **@{-} ?<*\dir{<},  (16,9);(24,0) **@{-},
(12.8,12.4);(14.3,10.8) **@{-},
(11.3,7.3);(9.5,9.3) **@{-},
(11.5,-5) *{\epsilon=1},
\endxy}
\vspace*{10pt}\caption{$T^{-\epsilon}_{i,j}$}
\label{fig-gamma-p-db}
\end{figure}

\begin{figure}[ht]
\begin{center}
\resizebox{0.60\textwidth}{!}{%
  \includegraphics{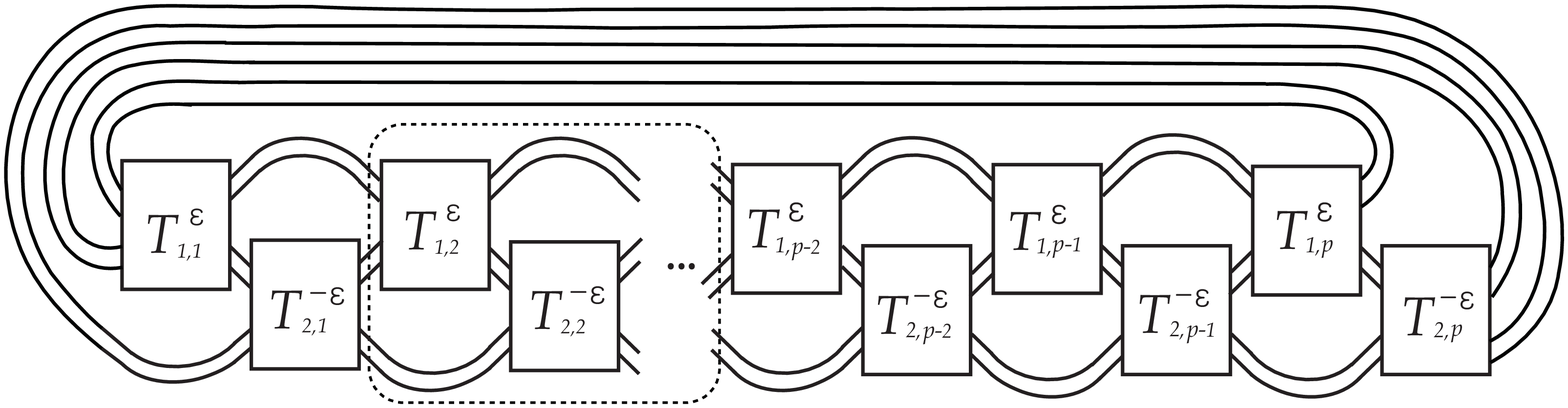} }
\caption{$W_{2}(\hat\gamma_{p})$}\label{tangle}
\end{center}
\end{figure}

For our convenience, we represent the standard diagram $W_{2}(\hat\gamma_{p})$ in Fig.~\ref{tangle} by the $2\times p$ matrix
\begin{equation*}
Q_p=\begin{pmatrix}
T_{1,1}^{\epsilon} & T_{1,2}^{\epsilon}&\cdots&T_{1,p-1}^{\epsilon}&T_{1,p}^{\epsilon}\\
T_{2,1}^{-\epsilon} & T_{2,2}^{-\epsilon}&\cdots&T_{2,p-1}^{-\epsilon}&T_{2,p}^{-\epsilon}
\end{pmatrix}.
\end{equation*}

In the case that $\epsilon=-1$, we will denote the diagram $W_{2}(\hat\gamma_{p})$ simply by $D_{p}$ and $N_{p}$ denote the integer given by
\begin{equation*}
N_{p}=c(D_{p})-s(D_{p})+1=8p-(4p+2)+1=4p-1~(p\geq 3).
\end{equation*}

In what follows, instead of the diagram $D_{p}$ illustrated in Fig.~\ref{tangle}, we use a shortcut diagram shown in Fig.~\ref{tangle1} for $D_{p}$ for the sake of simplicity.

\begin{figure}[ht]
\begin{center}
\resizebox{0.60\textwidth}{!}{%
  \includegraphics{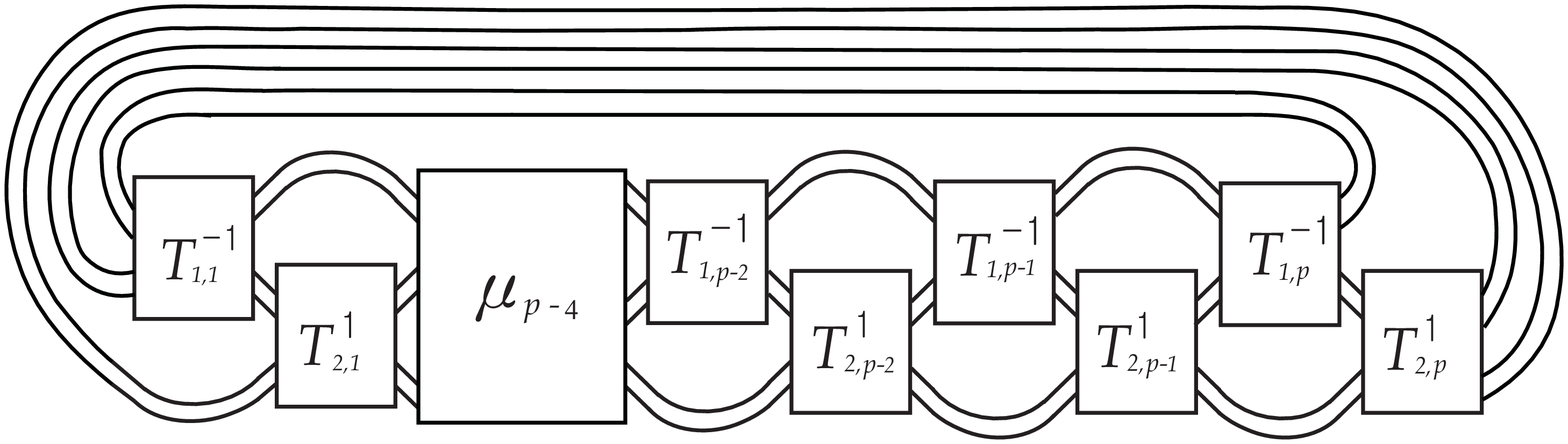} }
\caption{$D_{p}=W_{2}(\hat\gamma_{p})$ with $\epsilon=-1.$}\label{tangle1}
\end{center}
\end{figure}

\begin{example}\label{examp-p=2}
The closure $\hat\gamma_2$ of the $3$-braid $\gamma_2=(\sigma_2^{-1}\sigma_{1})(\sigma_2^{-1}\sigma_{1})$ is the figure-eight knot $4_{1}$ (see Fig. \ref{fig-eight}) and the doubled link $D_2=W_2(\hat\gamma_2)$ is represented by $2\times 2$ matrix
\[Q_{2}=\left(
\begin{array}{cc}
T^{-1}_{1,1}&T^{-1}_{1,2}\\
T^{1}_{2,1}&T^{1}_{2,2}
\end{array}
\right).
\]

\begin{figure}[ht]
\begin{center}
\resizebox{0.40\textwidth}{!}{%
  \includegraphics{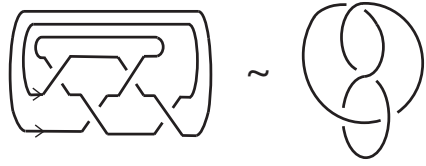} }
\caption{The figure-eight knot $\hat\gamma_2$}\label{fig-eight}
\end{center}
\end{figure}

By a direct computation, we obtain
\begin{align*}
P_{W_2(\hat\gamma_2)}(v,z)
&=z^{-1}(-v^{5}+3v^{3}-5v+5v^{-1}-3v^{-3}+v^{-5})\\
&+z(-2v^{5}+4v^{3}-4v^{-3}+2v^{-5})\\
&+z^{3}(-v^{5}+v^{3}+9v-9v^{-1}-v^{-3}+v^{-5})\\
&+z^{5}(6v-6v^{-1})+z^{7}(v-v^{-1}).
\end{align*}

Hence the maximal $z$-degree of the HOMFLYPT polynomial $P_{W_2(\hat\gamma_2)}(v,z)$ of the doubled link $D_2=W_2(\hat\gamma_2)$ is given
by \[\max\deg_z P_{W_2(\hat\gamma_2)}(v,z)=7=2\cdot 4-1=2c(\hat\gamma_2)-1.\]
On the other hand, let $\hat\gamma_2^*$ denote the mirror image of $\hat\gamma_2$. Then we also have
\begin{align*}
\max\deg_z P_{W_2(\hat\gamma_2^*)}(v,z)&=\max\deg_z P_{W_2(\hat\gamma_2)}(v^{-1},z)=7=2\cdot 4-1=2c(\hat\gamma_2^*)-1.
\end{align*}
\end{example}

Now we apply the partial skein tree in Fig.~\ref{tangle2} for the tangle $T_{2,p}^{1}$ in $D_{p}$ which is of the tangle in the left-hand side of Fig.~\ref{fig-gamma-p-db}. Let $D_{p}^{i}$ $(1\leq i \leq8)$ denote the link diagram represented by $2\times p$ matrix
\begin{equation*}
D_{p}^{i}=\begin{pmatrix}
T_{1,1}^{-1} & T_{1,2}^{-1}&\cdots&T_{1,p-1}^{-1}&T_{1,p}^{-1}\\
T_{2,1}^{1} & T_{2,2}^{1}&\cdots&T_{2,p-1}^{1}&T_{i}
\end{pmatrix}.
\end{equation*}

That is, $D_{p}^{i}$ is the link diagram obtained from the link diagram $D_{p}$ by replacing the tangle $T_{2,p}^{1}$
with the tangle $T_{i}$ as in (\ref{tang-assign}).
Hence two diagrams $D_{p}$ and $D_{p}^{i}$ are identical except for the tangle corresponding to the $(2,p)-$entry of the matrix notation. In these terminologies, we have the following Lemma~\ref{m-d-1} that will play an essential role in the proof of Lemma \ref{m-d-1-1} below.
\begin{lemma}\label{m-d-1}
For any integer $p\geq3$,
\begin{itemize}
\item [(1)] $\max\deg_{z}P_{D_{p}^{4}}(v,z)\leq N_{p}-3=4p-4$.
\item [(2)] $\max\deg_{z}P_{D_{p}^{5}}(v,z)\leq N_{p}-3=4p-4$.
\item [(3)] $\max\deg_{z}P_{D_{p}^{6}}(v,z)\leq N_{p}-3=4p-4$.
\item [(4)] $\max\deg_{z}P_{D_{p}^{7}}(v,z)\leq N_{p}-3=4p-4$.
\item [(5)] $\max\deg_{z}P_{D_{p}^{8}}(v,z)\leq N_{p}-4=4p-5$
\end{itemize}
\end{lemma}
The proof of this lemma \ref{m-d-1} will be given in the final section \ref{sect-pf-lem1}.

\begin{lemma}\label{m-d-1-1}
Let $W_{2}(\hat\gamma_{p})$ be the doubled link of the closure $\hat\gamma_{p}$ of the alternating $3$-braid $\gamma_{p}=(\sigma_2^{\epsilon}\sigma_1^{-\epsilon})^p$ with $\epsilon=\pm 1$. Then
\begin{align}\label{m-d-2}
\max\deg_{z}P_{W_{2}(\hat\gamma_{p})}(v,z)&=2c(\hat\gamma_{p})-1~ (p \geq 2)
\end{align}
\end{lemma}

\begin{proof} We prove the assertion (\ref{m-d-2}) by induction on $p$. If $p=2$, then $\gamma_2=(\sigma_2^{\epsilon}\sigma_1^{-\epsilon})^2$ whose closure is the figure eight knot and (\ref{m-d-2}) follows from Example \ref{examp-p=2}.

Now we assume that $p\geq3$ and (\ref{m-d-2}) holds for every integers $\leq p-1$. We consider two cases separately.

\smallskip

{\bf Case I.} $\epsilon=-1.$ In this case, we have $W_{2}(\hat\gamma_{p})=D_{p}$ by the notational convention above (see Fig.~\ref{tangle1}).

\smallskip

{\bf Claim.} $\max\deg_{z}P_{D_{p}}(v,z)=2c(\hat\gamma_{p})-1=4p-1.$

\smallskip

{\bf Proof of Claim.} From the skein relation for the HOMFLYPT polynomial and a partial skein tree for $T_{2,p}^{1}$ in Fig.~\ref{tangle2}, we obtain
\begin{align}\label{m-d-3}
&P_{D_{p}}(v,z) =(P_{D_{p}^{1}}(v,z)+ P_{D_{p}^{2}}(v,z)-P_{D_{p}^{3}}(v,z))z^{2}\\ &+
(vP_{D_{p}^{4}}(v,z)-v^{-1}P_{D_{p}^{5}}(v,z)+vP_{D_{p}^{6}}(v,z)-vP_{D_{p}^{7}}(v,z))z+P_{D_{p}^{8}}(v,z).\notag
\end{align}

Let $L'$ be the link represented by the standard braid diagram $\hat\gamma_{p-1}$, which is the closure of the alternating $3$-braid $\gamma_{p-1}=(\sigma_2^{-1}\sigma_1)^{p-1}$.
Then $L'$ is a non-split alternating link and so $c(L')=c(\hat\gamma_{p-1})=2(p-1)$. By induction hypothesis, we have
\begin{align}\label{m-d-4}
\max\deg_{z}P_{W_2(\hat\gamma_{p-1})}(v,z)=2c(\hat\gamma_{p-1})-1~(p\geq3).
\end{align}
Now let $D$ be the oriented link represented by the diagram obtained from the closed braid diagram $\hat\gamma_{p-1}$ by replacing a crossing in $\hat\gamma_{p-1}$ with a full-twist (so that $c(D)=c(\hat\gamma_{p-1})+1$) as illustrated in Fig.~\ref{tangle4}.

\begin{figure}[ht]
\begin{center}
\resizebox{0.60\textwidth}{!}{%
  \includegraphics{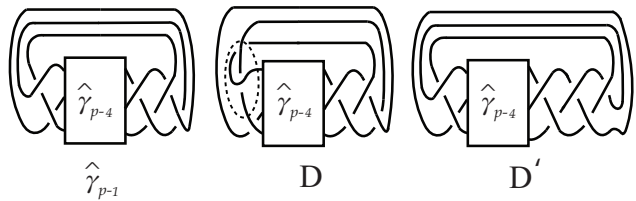} }
\caption{$\hat\gamma_{p-1}$ with a full-twist}\label{tangle4}
\end{center}
\end{figure}

By Proposition \ref{cr-1} and (\ref{m-d-4}), it follows that
\begin{align}\label{m-d-5}
\max\deg_{z}P_{W_{2}(D)}(v,z)&=2c(D)-1=\max\deg_{z}P_{W_{2}(\hat\gamma_{p-1})}(v,z)+2\notag\\
&=(2c(\hat\gamma_{p-1})-1)+2=4p-3~(p\geq3).
\end{align}

\begin{figure}[ht]
\begin{center}
\resizebox{0.35\textwidth}{!}{%
  \includegraphics{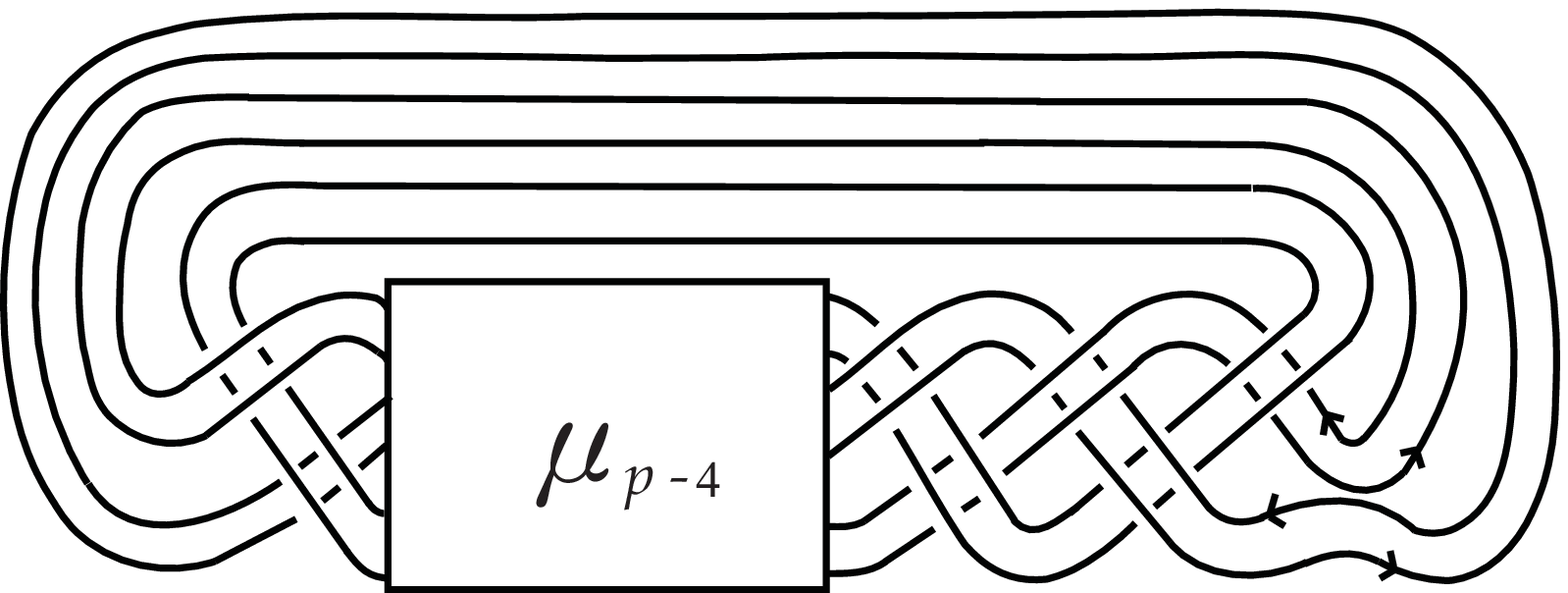} }
\caption{$D_{p}^{1}$}\label{tangle3}
\end{center}
\end{figure}

It is easily seen that the link diagram $D$ is isotopic to the link diagram $D^{'}$ in Fig.~\ref{tangle4}.
This shows that the link diagram $D_{p}^{1}$ (see Fig.~\ref{tangle3}) is just the doubled link diagram $W_2(D^{'})$. Hence we obtain from (\ref{m-d-5}) that
\begin{align}\label{m-d-5-1}
\max\deg_z P_{D_{p}^{1}}(v,z)&=\max\deg_z P_{W_2(D^{'})}(v,z)=\max\deg_z P_{W_2(D)}(v,z)\notag\\
&=4p-3~(p\geq3).
\end{align}

On the other hand, we observe that the link diagram $D_{p}^{2}$ is isotopic to the doubled link diagram in Fig.~\ref{tangle5}, which is precisely the doubled link diagram $W_2(D'')$, where $D''$ is the $2$-bridge link diagram of Conway normal form $C(1,1,\ldots,1,1,1,1,1)$.

\begin{figure}[ht]
\begin{center}
\resizebox{0.60\textwidth}{!}{%
  \includegraphics{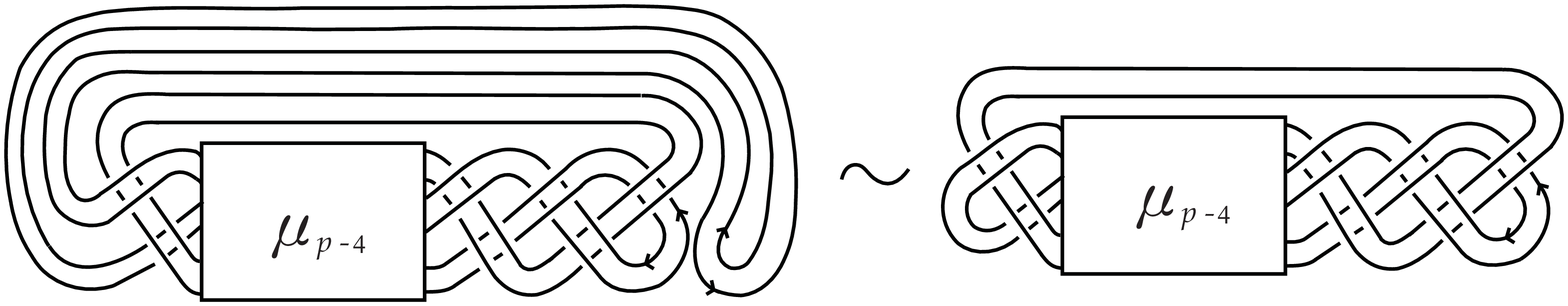} }
\caption{ $D_{p}^{2}$}\label{tangle5}
\end{center}
\end{figure}

Hence, by Proposition \ref{prop-Nakamura}, we have
\begin{align}\label{d-g-3-1}
\max\deg_{z}P_{D^{2}_{p}}(v,z)&=\max\deg_{z}P_{W_{2}(D'')}(v,z)
=2c(D'')-1=4p-3.
\end{align}

Since $\max\deg_z P_{D_{p}^3}(v,z)$ is too low to interfere with our main calculation by applying
Morton's inequality, we see that the maximum degree in $z$ for $P_{D_{p}^3}(v,z)$ does not contribute anything to $\max\deg_z P_{D_{p}}(v,z)$.
From (\ref{m-d-3}), (\ref{m-d-5-1}), (\ref{d-g-3-1}) and Lemma \ref{m-d-1}, we obtain that
\begin{equation*}
\max\deg_z P_{D_p}(v,z)= \max\{4p-1, 4p-3, 4p-5\}=4p-1=2c(\hat\gamma_{p})-1~(p\geq3).
\end{equation*}
This completes the proof of Claim.
Finally we obtain 
 \begin{equation*}
 \max\deg_{z}P_{W_{2}(\hat\gamma_{p})}(v,z)=
\max\deg_z P_{D_{p}}(v,z)=4p-1=2c(\hat\gamma_{p})-1.
\end{equation*}

\smallskip

{\bf Case II.} $\epsilon =1.$
It is easily seen that the corresponding link diagram $W_2(\hat\gamma_p)$ is just the mirror image of the diagram $D_p$ for which the assertion has already been established in the previous Case I. On the other hand, it is well known that if $L^*$ is the mirror image of an oriented link $L$, then  $P_{L^*}(v,z)=P_{L}(v^{-1},z)$. This fact implies that
$P_{W_2(\hat\gamma_p)}(v,z)=P_{D_p}(v^{-1},z)$. Hence
\begin{align*}
\max\deg_z P_{W_2(\hat\gamma_p)}(v,z)&=\max\deg_z P_{D_p}(v^{-1},z)=\max\deg_z P_{D_p}(v,z)=2c(\hat\gamma_p)-1.
\end{align*} 
This completes the proof of Lemma \ref{m-d-1-1}.
\end{proof}

Using Lemma \ref{m-d-1-1} and Proposition \ref{cr-1}, we obtain the following theorem which plays an important role in the proof of Theorem \ref{main-thm2} and Theorem \ref{main-thm3} of the next section \ref{sect-2-braid}.

\begin{theorem}\label{main-cor-1}
Let $\gamma_p (p \geq 2)$ be the alternating $3$-braid in (\ref{braid-beta}). If $L$ is a link having a diagram $D$ obtained from the closed braid diagram $\hat\gamma_p$ as shown in Fig.~\ref{fig-gamma-p} by replacing a crossing with a full twist (so that $c(D)=c(\hat\gamma_p)+1$), then \[\max\deg_z P_{W_2(D)}(v,z)=2c(D)-1.\]
\end{theorem}

\begin{proof}
Let $L'$ be the link represented by $\gamma_p$. It is obvious that $L'$ is a non-split alternating link with an alternating diagram $D'=\hat\gamma_p$ satisfying $c(L')=c(D')$. By Lemma \ref{m-d-1-1}, $\max\deg_z P_{W_2(D')}(v,z)=2c(D')-1$. Hence the assertion follows from Proposition \ref{cr-1}.
\end{proof}

%%%%%%%%%%%%%%%%%%%%%%%%%%%%%%%%%%%

\section{Proof of Theorems \ref{main-thm2} and \ref{main-thm3}}
\label{sect-2-braid}

\begin{proof}[{\bf Proof of Theorem \ref{main-thm2}.}]
Let $K$ be an alternating knot of braid index $3$, which is not the connected sum of two $(2, k)$-torus knots. By Theorem \ref{ch-alk-b3}, either $K$ is an alternating $3$-braid knot or a pretzel knot $\mathcal P(1,p,q,r)$ with $p,q,r \geq 1$.

First, if $K=\mathcal P(1,p,q,r)$, then it follows from \cite[Theorem 1]{BJ} that $g_c(W_{\pm}(K,m))$ $=1+p+q+r=c(K)$, establishing the assertion.

Now we assume that $K$ is an alternating $3$-braid knot. Then it is the closure $\hat\beta$ of an alternating $3$-braid:
\begin{equation*}
\beta=\sigma_1^{a_1}\sigma_2^{-b_1}\sigma_1^{a_2}\sigma_2^{-b_2}\sigma_1^{a_3}\cdots \sigma_2^{-b_{p-1}}\sigma_1^{a_p}\sigma_2^{-b_p} \in B_3,
\end{equation*}
where $p, a_i$ and $b_i$ are positive integers. Let $\eta=\sigma_1^{-a_1}\beta\sigma_1^{a_1}$. Then $K=\hat\beta=\hat\eta$ and
\begin{equation*}
\eta=\sigma_1^{-a_1}\beta\sigma_1^{a_1}=\sigma_2^{-b_1}\sigma_1^{a_2}\sigma_2^{-b_2}\sigma_1^{a_3}\cdots \sigma_2^{-b_{p-1}}\sigma_1^{a_p}\sigma_2^{-b_p}\sigma_1^{a_1}.
\end{equation*}
On the other hand, it is easily seen that the usual closed $3$-braid diagram $\hat\eta$ is obtained from the closed braid diagram $\hat\gamma_p$, where $\gamma_p=(\sigma_2^{-1}\sigma_1)^p$, by repeatedly replacing half-twists corresponding to the braid generators $\sigma_1$ and $\sigma_2^{-1}$ with full twists. Hence, by the corresponding repeated application of Theorem \ref{main-cor-1}, we obtain
\begin{equation}\label{pf-eqn1-thm13}
\max\deg_z P_{W_2(\hat\eta)}(v,z)=2c(\hat\eta)-1.
\end{equation}
It should be noted here that since at every stage the process of producing full twists builds an alternating connected diagram with no nugatory crossings, it follows that the underlying link is always a non-split alternating link diagram at every stage \cite{Men}.

Now, for any given integer $m$, let $W_\pm(K,m)$ be the $m$-twisted positive/negative Whitehead double of $K$ and let $W_\pm(\hat\eta,m)$ be the canonical diagram for $W_\pm(K,m)$ associated with the closed braid diagram $\hat\eta$. Since $c(\hat\eta) \geq 2p$, it follows from (\ref{pf-eqn1-thm13}) and Proposition \ref{prop3-cr-nbr-cg-wd} that $\max\deg_z P_{W_\pm(K,m)}(v,z) > 0$ and so $\max\deg_z P_{W_2(\hat\eta, w(\hat\eta))}(v,z)$ $\not= 1$. By Proposition \ref{prop3-cr-nbr-cg-wd}, we have
\begin{align}
 \max\deg_z P_{W_\pm(K,m)}(v,z)
 &=\max\deg_z P_{W_\pm(\hat\eta,m)}(v,z)\notag\\
 &=\max\deg_z P_{W_2(\hat\eta,m)}(v,z) +1\notag\\
 &=\max\deg_z P_{W_2(\hat\eta,w(\hat\eta))}(v,z) +1\notag\\
 &=\max\deg_z P_{W_2(\hat\eta)}(v,z)+1\notag\\
 &=2c(\hat\eta)-1+1=2c(K).\label{pf-eqn2-thm13}
\end{align}

Thus it follows from Proposition \ref{prop2-cr-nbr-cg-wd} and (\ref{pf-eqn2-thm13}) that
\begin{align*}
c(K)&=\frac{1}{2} \max\deg_z P_{W_\pm(K,m)}(v,z) \leq g_c(W_\pm(K,m))\notag\\ &\leq g_c(W_\pm(\hat\eta,m))=c(K).
\end{align*}
This gives $g_c(W_\pm(K,m))=c(K)$.

Finally, in the case that $K$ is the closure of the mirror image $\beta^*$ of the braid $\beta$, the same argument with $\gamma_p^*=(\sigma_2\sigma_1^{-1})^p$ gives $g_c(W_\pm(K,m))=c(K)$.
This competes the proof of Theorem \ref{main-thm2}.
\end{proof}

%%%
\bigskip

\begin{proof}[{\bf Proof of Theorem \ref{main-thm3}.}] Let $K$ be an alternating knot in $\overline{\mathcal K}_p$. Then $K$ has a diagram $D$ which is obtained from the diagram of the closed $3$-braid $\hat\gamma_p$ by repeatedly replacing a crossing by a full twist. By Theorem \ref{main-cor-1} and repeated application of Proposition \ref{cr-1}, we obtain
\begin{equation*}
\max\deg_z P_{W_2(D)}(v,z)=2c(D)-1.
\end{equation*}
Now, for any given integer $m$, let $W_\pm(K,m)$ be the $m$-twisted positive/negative Whitehead double of $K$ and let $W_\pm(D,m)$ be the canonical diagram for $W_\pm(K,m)$ associated with $D$. By the same argument as in the proof of Theorem \ref{main-thm2}, we obtain
\begin{align*}
 \max\deg_z P_{W_\pm(K,m)}(v,z)=2c(K)
\end{align*}
and therefore $g_c(W_\pm(K,m))=c(K)$. This competes the proof of Theorem \ref{main-thm3}.
\end{proof}

%%%

\begin{example}\label{example}
Let $A=(n_{i,j})_{1\leq i \leq 2; 1 \leq j \leq p}$ be an arbitrary given $2\times p$ integral matrix with $n_{ij}>0$, i.e.,
\[A=\left(
\begin{array}{cccc}
n_{1,1}&n_{1,2}&\cdots & n_{1,p}\\
n_{2,1}&n_{2,2}&\cdots & n_{2,p}\\
\end{array}
\right).\]
Let $K_A$ denote an oriented link in $S^3$ having a diagram $D_A$ in which each tangle labeled a non-zero integer $n_{i,j}$ denotes a vertical $n_{i,j}$ half-twists as shown in Fig.~\ref{fig6.1}. Then $K_A$ is obtained from the diagram of the closed $3$-braid $\hat\gamma_p=(\sigma_2^{-1}\sigma_1)^p$ by repeatedly replacing a crossing by a full twist and so $K_A \in \overline{\mathcal K}_p$.  Hence we obtain from Theorem \ref{main-thm3} that for any integer $m$,
\begin{equation*}
\max\deg_z P_{W_\pm(K_A, m)}(v,z)=2c(K_A)=2\sum_{i=1}^2\sum_{j=1}^p|n_{i,j}|
\end{equation*}
and \[g_c(W_\pm(K_A,m))=c(K_A)=\sum_{i=1}^2\sum_{j=1}^p|n_{i,j}|.\]
In particular, if all $n_{ij}$ are odd, then it follows from \cite[Theorem 12]{LS} that the braid index $b(K_A)$ of $K_A$ is given by
\[b(K_A)=\frac{1}{2}{\rm span}_vP_{K_A}(v,z)+1=3+\sum_{j=1}^p\frac{n_{1j}+n_{2j}-2}{2}.\]
Therefore the class $\overline{\mathcal K}_p$ in Theorem \ref{main-thm3} contains alternating knots with arbitrary large braid index $\geq 3$.
\end{example}

\begin{figure}[t]
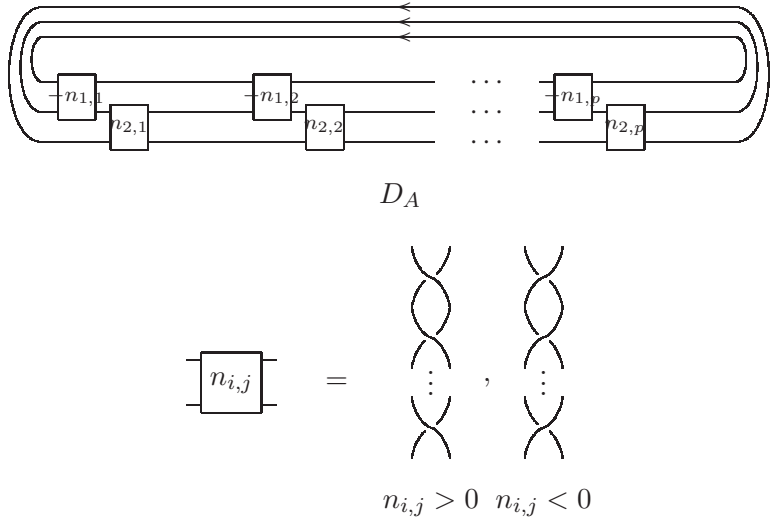

\vspace*{10pt} \centerline{\xy
 (75,16);(80,16) **@{-}, (75,22);(80,22) **@{-},
 (75,16);(75,22) **@{-}, (80,16);(80,22) **@{-},
 (77.5,19)*{_{n_{2,2}}},
 (68,20);(73,20) **@{-}, (68,26);(73,26) **@{-},
 (68,20);(68,26) **@{-}, (73,20);(73,26) **@{-},
 (70.5,23)*{_{-n_{1,2}}},
 %%%
 (66,17);(75,17) **@{-}, (80,17);(82,17) **@{-},
 (82,17);(92,17) **@{-},
 (66,21);(68,21) **@{-}, (73,21);(75,21) **@{-},
 (80,21);(92,21) **@{-},
 (66,25);(68,25) **@{-}, (73,25);(92,25) **@{-},
 %%%
 (56,17);(61,17) **@{-},
 (49,16);(54,16) **@{-}, (49,22);(54,22) **@{-},
 (49,16);(49,22) **@{-}, (54,16);(54,22) **@{-},
 (51.5,19)*{_{n_{2,1}}},
 (42,20);(47,20) **@{-}, (42,26);(47,26) **@{-},
 (42,20);(42,26) **@{-}, (47,20);(47,26) **@{-},
 (44.5,23)*{_{-n_{1,1}}},
 %%%
 (40,17);(49,17) **@{-}, (54,17);(56,17) **@{-},
 (61,17);(66,17) **@{-},
 (40,21);(42,21) **@{-}, (47,21);(49,21) **@{-},
 (54,21);(66,21) **@{-},
 (40,25);(42,25) **@{-}, (47,25);(66,25) **@{-},
 %%%
 (122,17);(127,17) **@{-},
 (115,16);(120,16) **@{-}, (115,22);(120,22) **@{-},
 (115,16);(115,22) **@{-}, (120,16);(120,22) **@{-},
 (117.5,19)*{_{n_{2,p}}},
 (108,20);(113,20) **@{-}, (108,26);(113,26) **@{-},
 (108,20);(108,26) **@{-}, (113,20);(113,26) **@{-},
 (110.5,23)*{_{-n_{1,p}}},
 %%%
 (106,17);(115,17) **@{-}, (120,17);(122,17) **@{-},
 (127,17);(132,17) **@{-},
 (106,21);(108,21) **@{-}, (113,21);(115,21) **@{-},
 (120,21);(132,21) **@{-},
 (106,25);(108,25) **@{-}, (113,25);(132,25) **@{-},
 %%%
 (40,31);(132,31) **@{-}, 
 (40,33);(132,33) **@{-},
 (40,35);(132,35) **@{-},
 (99,21)*{\cdots}, (99,25)*{\cdots},
 (99,17)*{\cdots},
 %%%
 (40,25);(40,31) **\crv{(38,25)&(38,31)},
 (40,21);(40,33) **\crv{(36,21)&(36,33)},
 (40,17);(40,35) **\crv{(34,17)&(34,35)},
 (132,25);(132,31) **\crv{(134,25)&(134,31)},
 (132,21);(132,33) **\crv{(136,21)&(136,33)},
 (132,17);(132,35) **\crv{(138,17)&(138,35)},
 %%%
 (87,31);(88,31) **@{} ?<*\dir{<},
 (87,35);(88,35) **@{} ?<*\dir{<},
 (87,33);(88,33) **@{} ?<*\dir{<}, (87.5,10)*{D_A},
 \endxy}
 \vskip 0.5cm
 \centerline{\xy (52,19);(60,19) **@{-},
(52,27);(60,27) **@{-}, (52,19);(52,27) **@{-}, (60,19);(60,27)
**@{-}, (50,20);(52,20) **@{-}, (50,26);(52,26) **@{-},
(60,20);(62,20) **@{-}, (60,26);(62,26) **@{-}, (56,23)*{n_{i,j}},
(70,23)*{=},
 %%%
 (80,41);(82.5,37) **\crv{(80,40)&(81,37.5)}, (85,33);(82.5,37) **\crv{(85,34)&(84,36.5)},
 (85,41);(83.3,37.5) **\crv{(85,40)&(84,38)}, (80,33);(81.7,36.5) **\crv{(80,34)&(81,36)},
 (80,33);(82.5,29) **\crv{(80,32)&(81,29.5)}, (85,25);(82.5,29) **\crv{(85,26)&(84,28.5)},
 (85,33);(83.3,29.5) **\crv{(85,32)&(84,30)}, (80,25);(81.7,28.5) **\crv{(80,26)&(81,28)},
 (80,21);(82.5,17) **\crv{(80,20)&(81,17.5)}, (85,13);(82.5,17) **\crv{(85,14)&(84,16.5)},
 (85,21);(83.3,17.5) **\crv{(85,20)&(84,18)}, (80,13);(81.7,16.5) **\crv{(80,14)&(81,16)},
 (82.5,24)*{\vdots}, (82.5,7)*{n_{i,j} > 0},
 %%%%
 (90,23)*{,},
 %%%%
 (100,41);(97.5,37) **\crv{(100,40)&(99,37.5)}, (95,33);(97.5,37) **\crv{(95,34)&(96,36.5)},
 (95,41);(96.7,37.5) **\crv{(95,40)&(96,38)}, (100,33);(98.3,36.5) **\crv{(100,34)&(99,36)},
 (100,33);(97.5,29) **\crv{(100,32)&(99,29.5)}, (95,25);(97.5,29) **\crv{(95,26)&(96,28.5)},
 (95,33);(96.7,29.5) **\crv{(95,32)&(96,30)}, (100,25);(98.3,28.5) **\crv{(100,26)&(99,28)},
 (100,21);(97.5,17) **\crv{(100,20)&(99,17.5)}, (95,13);(97.5,17) **\crv{(95,14)&(96,16.5)},
 (95,21);(96.7,17.5) **\crv{(95,20)&(96,18)}, (100,13);(98.3,16.5) **\crv{(100,14)&(99,16)},
 (97.5,24)*{\vdots}, (97.5,7)*{n_{i,j} < 0}, 
\endxy
 }
\vspace*{5pt}\caption{Diagram $D_A$ of $K_A$}\label{fig6.1}
\end{figure}

%%%%%%%%%%%%%%%%%%%%%%%%%%%%%%%%%%%%%

\section{Proof of Lemma \ref{m-d-1}} \label{sect-pf-lem1}

In this final section, we prove Lemma \ref{m-d-1}. For this purpose, we first remind the reader Lemma \ref{m-d-1}. Recall that $D_p$ denotes the doubled link $W_2(\hat\gamma_p)$ corresponding to the matrix notation $Q_p$ with $\epsilon=-1$ and $D_{p}^i$ ($4 \leq i \leq 8$) denotes the link diagram obtained from $D_p$ by replacing $T^{1}_{2,p}$ with $T_i$, where $T_4=F_1, T_5=F_2, T_6=F_3, T_7=F_4, T_8=G$ (see Section \ref{sect-miwd3-plrq}).

\vskip 0.3cm

\noindent{\bf Lemma \ref{m-d-1}.}
For any integer $p\geq3$,
\begin{itemize}
\item [(1)] $\max\deg_{z}P_{D_{p}^{4}}(v,z)\leq N_{p}-3=4p-4$.
\item [(2)] $\max\deg_{z}P_{D_{p}^{5}}(v,z)\leq N_{p}-3=4p-4$.
\item [(3)] $\max\deg_{z}P_{D_{p}^{6}}(v,z)\leq N_{p}-3=4p-4$.
\item [(4)] $\max\deg_{z}P_{D_{p}^{7}}(v,z)\leq N_{p}-3=4p-4$.
\item [(5)] $\max\deg_{z}P_{D_{p}^{8}}(v,z)\leq N_{p}-4=4p-5$
\end{itemize}

\begin{proof} (1) Consider a partial skein tree for $D_{p}^{4}~ (p\geq 3)$ and isotopy deformations as shown in Fig.~\ref{tangle8}, which yields the identity:
\begin{equation*}
P_{D_{p}^{4}}(v,z)=-v^{-1}zP_{a_{1}}(v,z)+v^{-2}P_{a_{2}}(v,z).
\end{equation*}

\begin{figure}[ht]
\begin{center}
\resizebox{0.75\textwidth}{!}{%
  \includegraphics{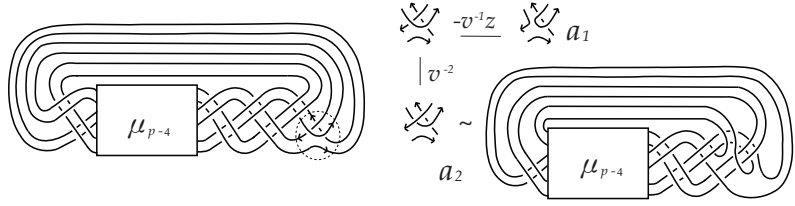} }
\caption{A partial skein tree for $D_{p}^{4}$}
\label{tangle8}
\end{center}
\end{figure}

\begin{figure}[ht]
\begin{center}
\resizebox{0.50\textwidth}{!}{%
  \includegraphics{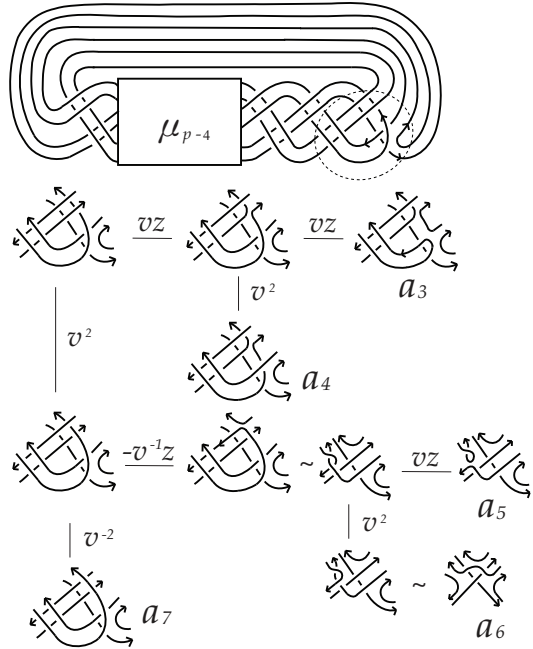} }
\caption{A partial skein tree for $D_{p}^{5}$}\label{tangle9}
\end{center}
\end{figure}

It is clear from Fig.~\ref{tangle8} that the link $a_{1}$ does not contribute anything to\\ $\max\deg_zP_{D_{p}^{4}}(v,z)$. For the links $a_2$, it follows from Morton's inequality that
\begin{align*}
\max\deg_{z}P_{a_{2}}(v,z)&\leq c(a_{2})-s(a_{2})+1\notag\\
&\leq (c(D_{p})-4)-(s(D_{p})-1)+1=N_{p}-3.
\end{align*}
This completes the proof of (1).

\bigskip

(2) From a partial skein tree for $D_{p}^{5}$ as shown in Fig.~\ref{tangle9}, we get
\begin{align}\label{d-g-8}
P_{D_{p}^{5}}(v,z)&=v^{2}z^{2}P_{a_{3}}(v,z)+v^{3}zP_{a_{4}}(v,z)-v^{2}z^{2}P_{a_{5}}(v,z)\notag\\
&~~~~-v^{3}zP_{a_{6}}(v,z)+P_{a_{7}}(v,z).
\end{align}
It is quite easy to see that the link $a_{3}$ and $a_{5}$ do not contribute anything to $\max\deg_zP_{D_{p}^{5}}(v,z)$. Let $a'_4$ be a diagram obtained from $a_4$ by isotopy as illisutrated in Fig.~\ref{fig-a4-1}. Then, by Morton's inequality, we obtain
\begin{align}\label{d-g-9}
\max\deg_{z}P_{a_{4}}(v,z)&\leq c(a'_{4})-s(a'_{4})+1\notag\\
&\leq (c(D_{p})-6)-(s(D_{p})-2)+1=N_{p}-4.
\end{align}

\begin{figure}[ht]
\begin{center}
\resizebox{0.80\textwidth}{!}{%
  \includegraphics{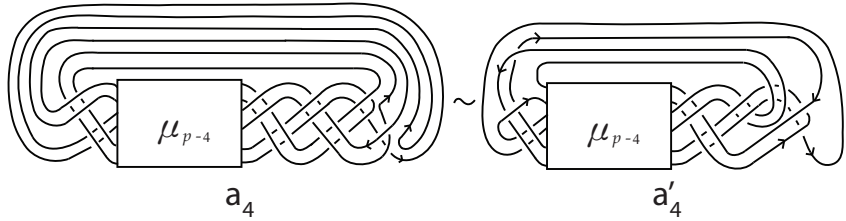} }
\caption{Another diagram $a'_4$ of $a_{4}$}\label{fig-a4-1}
\end{center}
\end{figure}

\begin{figure}[ht]
\begin{center}
\resizebox{0.45\textwidth}{!}{%
  \includegraphics{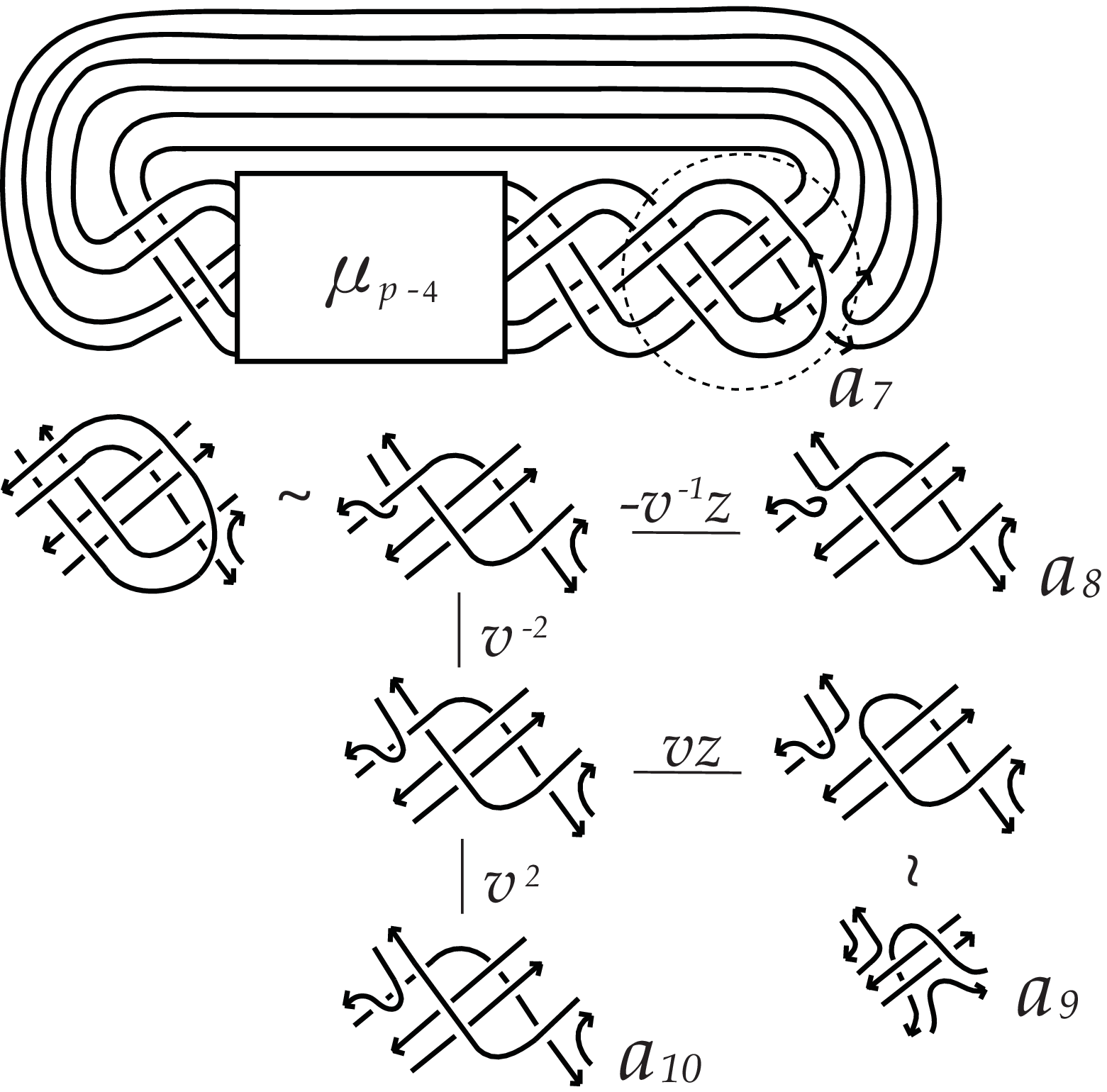} }
\caption{A partial skein tree for $a_{7}$}\label{tangle10}
\end{center}
\end{figure}

\begin{figure}[ht]
\begin{center}
\resizebox{0.40\textwidth}{!}{%
  \includegraphics{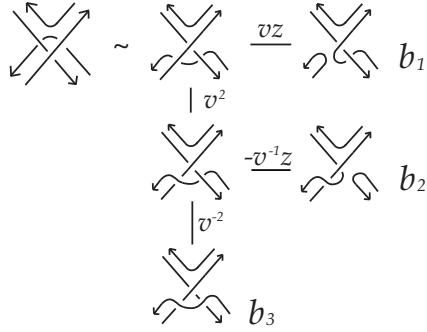} }
\caption{A partial skein tree for $D_{p}^{6}$}\label{tangle11}
\end{center}
\end{figure}
For the link $a_{6}$, we have
\begin{align}
\max\deg_{z}P_{a_{6}}(v,z)&\leq c(a_{6})-s(a_{6})+1\notag\\
&\leq (c(D_{p})-9)-(s(D_{p})-3)+1=N_{p}-6.\label{d-g-10}
\end{align}

For the link $a_{7}$, we obtain from Fig.~\ref{tangle10} that
\begin{align}\label{d-g-11}
P_{a_{7}}(v,z)&=-v^{-1}zP_{a_{8}}(v,z)+v^{-1}zP_{a_{9}}(v,z)+P_{a_{10}}(v,z).
\end{align}
Clearly, the link $a_{8}$ does not contribute anything to $\max\deg_zP_{a_{7}}(v,z)$ and so by Morton's inequality,
\begin{align}\label{d-g-12}
\max\deg_{z}P_{a_{9}}(v,z)&\leq c(a_{9})-s(a_{9})+1\notag\\
&\leq (c(D_{p})-12)-(s(D_{p})-6)+1=N_{p}-6,\\
\max\deg_{z}P_{a_{10}}(v,z)&\leq c(a_{10})-s(a_{10})+1\notag\\
&\leq (c(D_{p})-16)-(s(D_{p})-7)+1=N_{p}-9.\label{d-g-13}
\end{align}
Therefore we have from (\ref{d-g-11}), (\ref{d-g-12}) and (\ref{d-g-13}) that
\begin{align}\label{d-g-14}
\max\deg_{z}P_{a_{7}}(v,z)&\leq \max\{M(a_{9})+1, M(a_{10})\}\notag\\
&\leq \max\{N_{p}-5, N_{p}-9\}=N_{p}-5,
\end{align}
From (\ref{d-g-8}), (\ref{d-g-9}), (\ref{d-g-10}) and (\ref{d-g-14}), we have
\begin{align*}
\max\deg_{z}P_{D_{p}^{5}}(v,z)&\leq \max\{M(a_{4})+1, M(a_{6})+1, M(a_{7})\}\notag\\
&\leq \max\{N_{p}-3, N_{p}-5, N_{p}-5\}=N_{p}-3.
\end{align*}
This completes the proof of (2).

\bigskip

(3) We consider two cases separately.

\smallskip

{\bf Case 1.} If $p=3k~(k=1,2,\cdots)$, then the closed braid $\widehat{\gamma_{p}}$ is an oriented link of three components. From the skein relation for the HOMFLYPT polynomial and a partial skein tree for $D_{p}^{6}$ in Fig.~\ref{tangle11}, we obtain
\begin{align*}
P_{D_{p}^{6}}(v,z)&=(vP_{b_{1}}(v,z)-vP_{b_{2}}(v,z))z+P_{b_{3}}(v,z).
\end{align*}
Clearly, the link $b_{1}$ and  $b_{2}$ do not contribute anything to $\max\deg_zP_{D_{p}^{6}}(v,z)$.
Moving two crossings of $b_3$ labeled $1, 2$ to the place labeled $3, 4$ in $b_3$, respectively, along $2$-parallel strings by isotopy, we obtain the diagram $b_4$, which is isotopic to the diagram $b_5$ as illustrated in Fig.~\ref{tangle12}. Now we switch parallel strings of the other components in $b_{5}$ which do not incident to the crossings labeled $1, 2$ as illustrated in Fig.~\ref{tangle13} by isotopy. Then the resulting diagram, also denoted by $b_5$, is indeed $-D^4_p$, that is, the diagram $D^4_p$ with reverse orientations for all components. Hence it follows from (1) that \begin{align}
\max\deg_{z}P_{D_{p}^{6}}(v,z)&=\max\deg_{z}P_{b_{5}}(v,z)=\max\deg_{z}P_{-D_{p}^{4}}(v,z)\notag\\
&=\max\deg_{z}P_{D_{p}^{4}}(v,z)\leq N_{p}-3.\label{eq-tangle11}
\end{align}

\begin{figure}[ht]
\begin{center}
\resizebox{0.30\textwidth}{!}{%
  \includegraphics{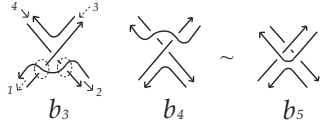} }
\caption{An isotopy deformation of $b_3$}\label{tangle12}
\end{center}
\end{figure}

\begin{figure}[ht]
\begin{center}
\resizebox{0.35\textwidth}{!}{%
  \includegraphics{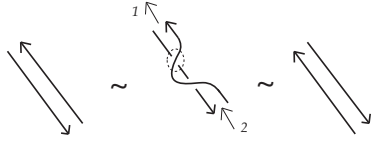} }
\caption{Switching parallel strings}\label{tangle13}
\end{center}
\end{figure}

{\bf Case 2.} If $p=3k+1$ or $p=3k+2$ $(k=1,2,\cdots)$, then $\widehat{\gamma_{p}}$ is an oriented knot.
\begin{figure}[ht]
\begin{center}
\resizebox{0.40\textwidth}{!}{%
  \includegraphics{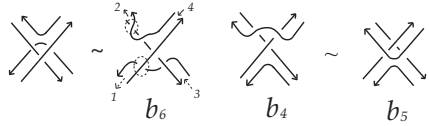} }
\caption{An isotopy deformation of $D_{p}^{6}$}\label{tangle14}
\end{center}
\end{figure}

We first observe that $D^6_p$ is clearly isotopic to the diagram $b_6$ in  Fig.~\ref{tangle14}. Moving two crossings of $b_6$ labeled $1, 2$ to the place labeled $3, 4$ in $b_6$, respectively, along $2$-parallel strings, we obtain the diagram $b_{4}$ which is isotopic to $b_5$ illustrated in Fig.~\ref{tangle14}. Since $b_5$ is just $-D^4_p$, we have (\ref{eq-tangle11}).
This completes the proof of (3).

\bigskip

(4) We consider two cases separately.

\bigskip

{\bf Case 1.} If $p=3k~(k=1,2,\cdots)$, then the closed braid $\widehat{\gamma_{p}}$ is an oriented link of three components. From the skein relation for the HOMFLYPT polynomial and a partial skein tree for $D_{p}^{7}$ in  Fig.~\ref{tangle15}, we obtain
\begin{align*}
P_{D_{p}^{7}}(v,z)&=-v^{-1}zP_{b_{7}}(v,z)+v^{-2}P_{b_{8}}(v,z).
\end{align*}
\begin{figure}[ht]
\begin{center}
\resizebox{0.40\textwidth}{!}{%
  \includegraphics{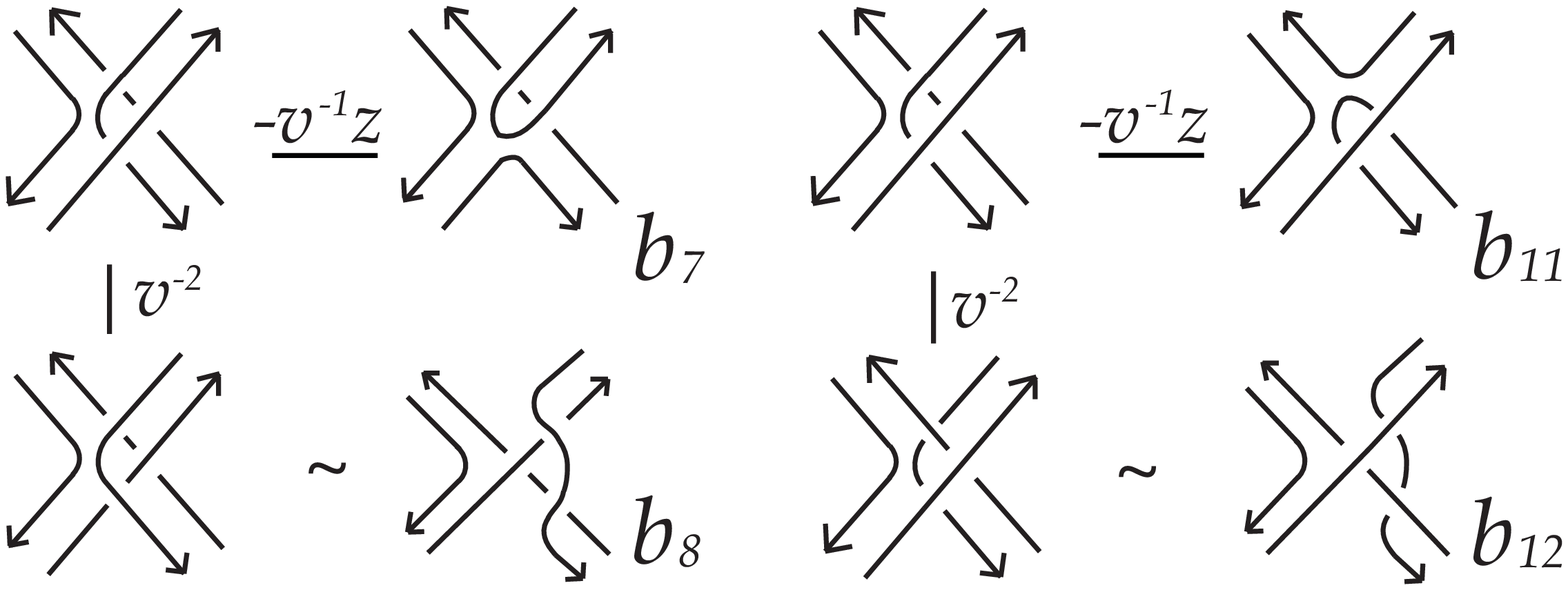} }
\caption{A partial skein tree for $D_{p}^{7}$}\label{tangle15}
\end{center}
\end{figure}

Observe that the link $b_{7}$ does not contribute anything to $\max\deg_zP_{D_{p}^{7}}(v,z)$. Moving two crossings of $b_8$ labeled $1, 2$ to the place labeled $3, 4$, respectively, along $2$-parallel strings by isotopy, we obtain the diagram $b_9$, which is isotopic to the diagram $b_{10}$ as illustrated in Fig.~\ref{tangle16}. Now we switch parallel strings of the other components in $b_{10}$ that are not incident to the crossings labeled $1, 2$ by isotopy. Then the resulting diagram is just $-D^5_p$. Hence it follows from (2) that \begin{align}
\max\deg_{z}P_{D_{p}^{7}}(v,z)&=\max\deg_{z}P_{b_{8}}(v,z)=\max\deg_{z}P_{-D_{p}^{5}}(v,z)\notag\\
&=\max\deg_{z}P_{D_{p}^{5}}(v,z)\leq N_{p}-3.\label{eq-tangle11-1}
\end{align}
\begin{figure}[ht]
\begin{center}
\resizebox{0.30\textwidth}{!}{%
  \includegraphics{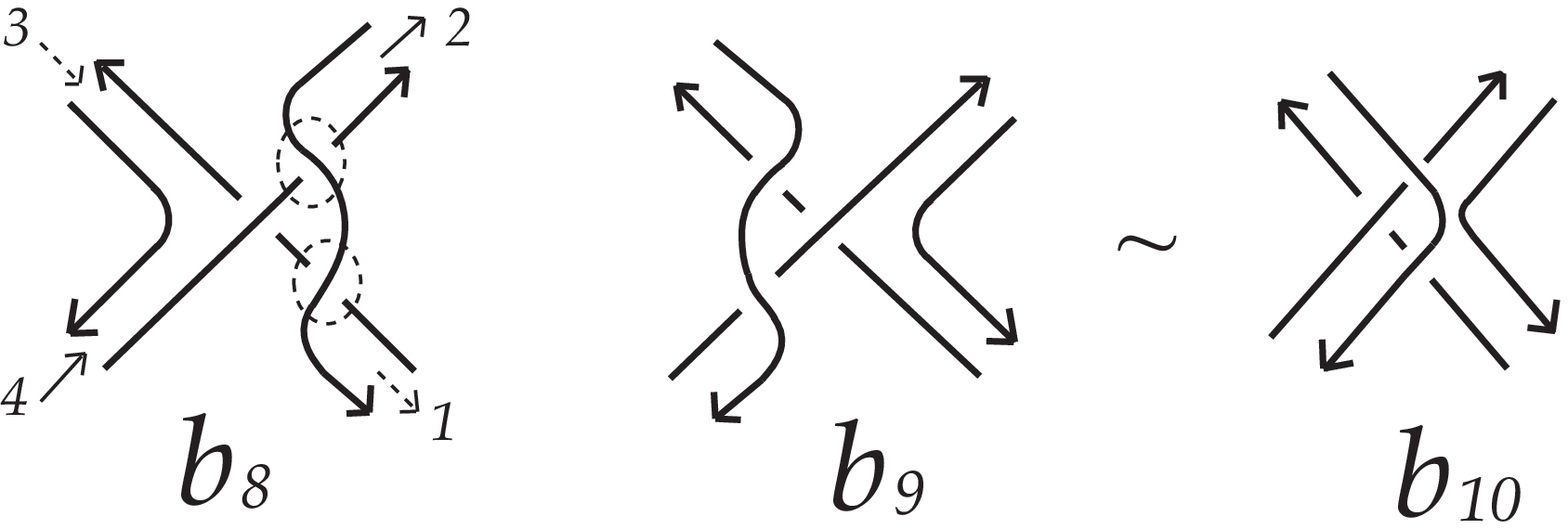} }
\caption{An isotopy deformation of $b_8$}\label{tangle16}
\end{center}
\end{figure}

\begin{figure}[ht]
\begin{center}
\resizebox{0.30\textwidth}{!}{%
  \includegraphics{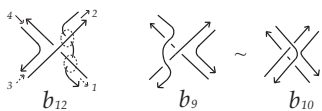} }
\caption{An isotopy deformation of $b_{12}$}\label{tangle18}
\end{center}
\end{figure}

\begin{figure}[ht]
\begin{center}
\resizebox{0.55\textwidth}{!}{%
  \includegraphics{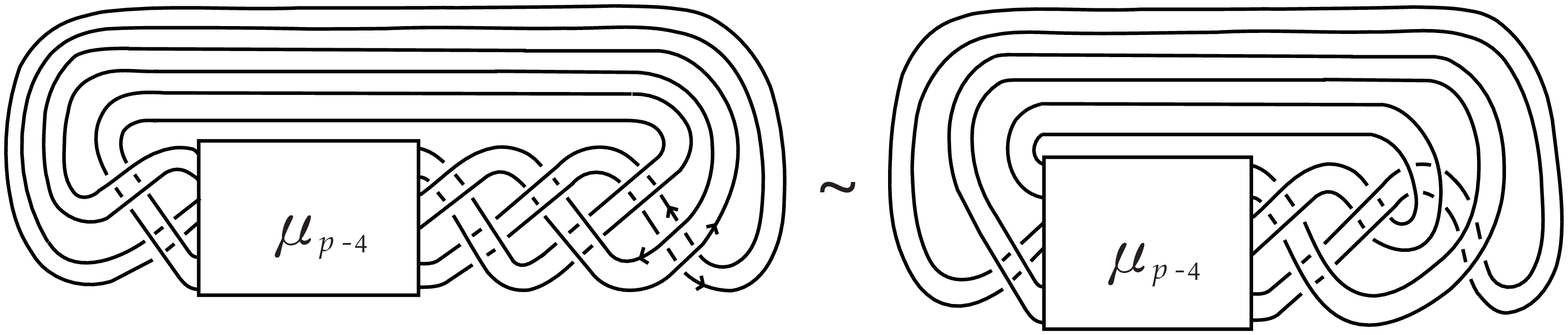} }
\caption{$D_{p}^{8}$}\label{tangle19}
\end{center}
\end{figure}

\smallskip

{\bf Case 2.} If $p=3k+1$ or $p=3k+2$ $(k=1,2,\cdots)$, then $\widehat{\gamma_{p}}$ is an oriented knot. From the skein relation for the HOMFLYPT polynomial and a partial skein tree for $D_{p}^{7}$ in Fig.~\ref{tangle15}, we obtain
\begin{align*}
P_{D_{p}^{7}}(v,z)&=-v^{-1}zP_{b_{11}}(v,z)+v^{-2}P_{b_{12}}(v,z).
\end{align*}
It is clear that the link $b_{11}$ does not contribute anything to $\max\deg_zP_{D_{p}^{7}}(v,z)$. Now, moving two crossings of $b_{12}$ labeled $1, 2$ to the place labeled $3, 4$, respectively, along $2$-parallel strings, we obtain the diagram $b_{9}$ which is isotopic to $b_{10}$ as illustrated in Fig.~\ref{tangle18}. Since $b_{10}$ is just $-D^5_p$ as above, we then have (\ref{eq-tangle11-1}).  This completes the proof of the assertion (4).

\bigskip

(5) It follows from Fig.~\ref{tangle19} and Morton's inequality that
\begin{align*}
{\rm maxdeg}_{z}P_{D_{p}^{8}}(v,z)&\leq c(D_{p}^{8})-s(D_{p}^{8})+1\notag\\
&\leq (c(D_{p})-8)-(s(D_{p})-4)+1=N_{p}-4.
\end{align*}
This completes the proof of the assertion (5), and so completes the proof of Lemma \ref{m-d-1}.
\end{proof}

%%%%%%%%%%%%%%%%%%%%%%%%%%%%%%%%%%%%%%%%%%%%%%%%%%%%

\section*{Acknowledgments}
This research was supported by Basic Science Research Program through the National Research Foundation of Korea(NRF) funded by the Ministry of Education, Science and Technology(2012R1A6A3A01012801). The authors would like to thank Professor K. Tanaka for his valuable comments.

%%%


\begin{thebibliography}{9}

\bibitem{BM} {J. S. Birman and W. W. Menasco}, A note on closed $3$-braids, \textit{Commun. Contemp. Math.} \textbf{10} (2008), 1033--1047.

\bibitem{BJ} {M. Brittenham and J. Jensen}, Canonical genus and the Whitehead doubles of pretzel knots, \textit{arXiv:math.GT/0608765 v1.}

\bibitem{BZ}  {G. Burde and H. Zieschang}, \textit{Knots,} (Walter de Gruyter \& Co., Berlin, 2003).

\bibitem{FP} {F. Frankel and L. Pontrjagin}, Ein Knotensatz mit Anwendung auf die Dimensionstheorie,  \textit{Math. Ann.} \textbf{102} (1930), 785-789.

\bibitem{Gr1} {H. Gruber}, On knot polynomials of annular surfaces and their boundary links,\textit{Math. Proc. Cambridge Philos. Soc.} \textbf{147} (2009), 173–183.

\bibitem{Hj} {H. J. Jang and S. Y. Lee}, The canonical genus for Whitehead doubles of a family of alternating knots, \textit{Topology and its Applications} {\textbf 159} (2012), 3563-3582.

\bibitem{KaMi} {T. Kanenobu and Y. Miyazawa}, 2-bridge link projections, \textit{Kobe J. Math.} {\textbf 9} (1992), 171-182.

\bibitem{Kaw} {A. Kawauchi}, \textit{A survey of knot theory,} (Birkh\"auser, 1996).

\bibitem{KK} {M. Kobayashi and T. Kobayashi}, On canonical genus and free genus of knot, \textit{J. Knot Theory Ramifications} {\textbf 5} (1996), 77--85.

\bibitem{LPS} {S. Y. Lee, C.-Y. Park and M. Seo}, On adequate links and homogeneous links, \textit{Bull. Austral. Math. Soc.} {\textbf 64} (2001), 395–404.

\bibitem{LeS3} {S. Y. Lee and M. Seo}, The genus of periodic links with rational quotients, \textit{Bull. Aust. Math. Soc.} \textbf{79} (2009), 273–284.

\bibitem{LS} {S. Y. Lee and M. Seo}, A formula for the braid index of links, \textit{Topology and its Applications} {\textbf 157} (2010), 247-260.

\bibitem {Liv} {C. Livingston}, The free genus of doubled knots, \textit{Proc. Amer. Math. Soc.} {\textbf 104} (1988), 329--333.

\bibitem{Man} {V. O. Manturov}, A combinatorial representation of links by quasitoric braids, \textit{European J. Combinatorics} {\textbf 23} (2002), 203--212.

\bibitem{Men} {W. Menasco}, Closed incompressible surfaces in alternating knot and link complements, \textit{Topology} {\textbf 23} (1984), 37--44.

\bibitem{Mor} {Y. Moriah}, On the free genus of knots, \textit{Proc. Amer. Math. Soc.} {\textbf 99} (1987), 373--379.

\bibitem{Mot} {H. R. Morton}, Seifert circles and knot polynimials, \textit{Math. Proc. Camb. Phil. Soc.} {\textbf 99} (1986), 107--109.

\bibitem{Mur} {K. Murasugi}, On the genus of the alternating knot I, II, \textit{J. Math. Soc. Japan} \textbf{10} (1958), 94-105, 235-248.

\bibitem{Mur2} {K. Murasugi}, On closed $3$-braids, \textit{Memories Amer. Math. Soc.} No. \textbf{151}, Providence, Rh. I.: Amer. Math. Soc., 1974.

\bibitem{Nak} {T. Nakamura}, On the crossing number of a $2$-bridge knot and the canonical genus of its Whitehead double, \textit{Osaka J. Math.} \textbf{43} (2006), 609-623.

\bibitem{Sei} {H. Seifert}, \"Uber das Geschlecht von Knoten, \textit{Math. Ann.} \textbf{110} (1936) 571-592.

\bibitem{Sch} {H. Schubert}, Knoten mit zwei Br\"ucken, \textit{Math Z.} {\textbf 65} (1956), 133-170.

\bibitem{Sto}  {A. Stoimenow}, The skein polynomial of closed $3$-braids, \textit{J. Reine Angew. Math.}  {\textbf 564}  (2003), 167–180. 

\bibitem{Tri} {J. Tripp}, The canonical genus of Whitehead doubles of a family torus knots, \textit{J. Knot Theory Ramifications} {\textbf 11} (2002), 1233-1242.
\end{thebibliography}
\end{document}